\documentclass{amsart}

\usepackage{amsmath, graphicx, cite, enumerate, mathrsfs, comment, color}
\usepackage{amssymb, amsthm, amscd, cancel}
\usepackage[paper=a4paper,headheight=0pt,left=3cm,top=2.5cm,right=3cm,bottom=2.5cm]{geometry}

\newtheorem{theorem}{Theorem}[section]

\newtheorem{lemma}[theorem]{Lemma}
\newtheorem{proposition}[theorem]{Proposition}
\newtheorem{corollary}[theorem]{Corollary}

\theoremstyle{definition}
\newtheorem{definition}[theorem]{Definition}

\theoremstyle{remark}
\newtheorem{remark}[theorem]{Remark}

\numberwithin{equation}{section}

\DeclareMathAlphabet{\mathpzc}{OT1}{pzc}{m}{it}

\newcommand{\abs}[1]{\left|#1\right|}
\newcommand{\tr}{\textup{tr}}

\newcommand{\Ric}{\textup{Ric}}

\newcommand{\Tr}{\textup{Tr}}
\newcommand{\id}{\textup{id}}

\newcommand{\lot}{\textup{l.o.t.}}

\newcommand{\R}{\mathbb{R}}

\renewcommand{\div}{\,\textup{div}}

\newcommand{\D}[2]{\frac{\partial #1}{\partial #2}}

\begin{document}
\title[]{A Spinorial Perelman's Functional: critical points and gradient flow}
\author{Tsz-Kiu Aaron Chow, Frederick Tsz-Ho Fong}
\email{chowtka@ust.hk}
\email{frederick.fong@ust.hk}
\address{Department of Mathematics, Hong Kong University of Science and Technology, Clear Water Bay, Kowloon, Hong Kong SAR.}
\date{5 January, 2026}
\maketitle

\begin{abstract}
In this article, we introduce an energy functional on closed Riemannian spin manifolds which unifies Perelman's $\mathcal{W}$- and $\mathcal{F}$-functionals, Baldauf-Ouzch's $\mathcal{E}$-functional, and the Dirichlet energy for spinors. We compute its first variation formula, and show that its critical points under natural constraints are twisted Ricci solitons and eigen-spinors of the weighted Dirac operator. We introduce a negative $L^2$-gradient flow of such a functional, and establish its short-time existence and uniqueness via contraction mapping methods.
\end{abstract}

\section{Introduction}
\subsection{Background and motivations}
Perelman's $\mathcal{F}$- and $\mathcal{W}$-functionals were introduced in his first preprint \cite{Perelman} on the proof of the Poincar\'e conjecture. These functionals play a central role in the analysis of singularities in the Ricci flow. Under suitable normalization constraints, their minimizers - such as the $\lambda$- and $\mu$-functionals - are monotone increasing along the flow and remain constant precisely when the metric is Ricci-flat, Einstein, or a Ricci soliton.

As explained in Remark 1.3 of \cite{Perelman}, Perelman's functionals are closely related to the Weitzenb\"ock (or  Bochner-Lichnerowicz) formula in the context of spin geometry. To further elaborate on this relation, we define the following quantities. Given a smooth scalar function $f : M \to \R$ on a Riemannian spin manifold $(M^n, g)$, for any tensor field $\alpha$ and scalar function $\varphi$, we define:
\begin{align*}
\div_f \alpha & = e^f \div(e^{-f}\alpha) = \div(\alpha) - i_{\nabla f}\alpha\\
D_f\psi & = e^{f/2} e_i \cdot \nabla_{e_i}(e^{-f/2}\psi) = D\psi - \frac{1}{2}\nabla f \cdot \psi\\
\Delta_f \varphi & = \div_f \nabla\varphi = \Delta\varphi - \langle \nabla f, \nabla\varphi\rangle\\
\Delta_f \psi & = -\nabla^*\nabla\psi - \nabla_{\nabla f}\psi \\
\Ric_f & = \Ric + \nabla\nabla f\\
R_f  & = R + 2\Delta f - \abs{\nabla f}^2 = R + \abs{\nabla f}^2 - 2\Delta_f f.
\end{align*}
Here we use $\nabla$ to denote both the Levi-Civita connection induced by $g$, and the spin connection compatible with $g$. Using this notation, we can express Perelman's $\mathcal{F}$-functional as
\[\mathcal{F}(g,f) = \int_M R_f e^{-f}\,d\mu_g.\]
By the Weitzenb\"ock's formula:
\[D_f^2 \psi = -\Delta_f \psi + \frac{1}{4}R_f \psi,\]
one can then use integration by parts to show:
\begin{align}
\label{eq:BO-function}
& \int_M 4\abs{D_f\psi}^2 e^{-f}\,d\mu_g - \mathcal{F}(g,f) \\
& = \int_M \left\{4\textup{Re}\langle D_f^2\psi, \psi\rangle - R_f\right\} e^{-f}\,d\mu_g \nonumber \\
& = \int_M \left\{4\abs{\nabla\psi}_g^2 + R_f\big(\abs{\psi}^2 - 1\big) \right\}e^{-f}\,d\mu_g \nonumber,
\end{align}
Hence, by considering the minimizing constraints below (with $g$ and $f$ fixed), we get
\begin{align*}
& \inf_{\int_M \abs{\psi}^2 e^{-f}\,d\mu_g = 1}\int_M 4\abs{D_f\psi}^2 e^{-f}\,d\mu_g - \mathcal{F}(g,f) \\
& = \inf_{\int_M \abs{\psi}^2 e^{-f}\,d\mu_g = 1}\int_M \left\{4\abs{\nabla\psi}_g^2 + R_f\big(\abs{\psi}^2 - 1\big) \right\}e^{-f}\,d\mu_g
\end{align*}
Recall that $D_f\psi = e^{f/2} D(e^{-f/2}\psi)$, so by the change of variables $\psi \mapsto e^{-f/2}\psi$ one can get
\[\inf_{\int_M \abs{\psi}^2 e^{-f}d\mu_g = }\int_M 4\abs{D_f\psi}^2 e^{-f}d\mu_g = \inf_{\int_M\abs{\psi}^2\,d\mu_g = 1}\int_M 4\abs{D\psi}^2\,d\Omega_0 = 4\lambda_1(D)^2\]
where $\lambda_1(D)$ is the lowest eigenvalue of the Dirac operator $D$, and so we have
\[4\lambda_1(D)^2 - \mathcal{F}(g,f) = \inf_{\int_M \abs{\psi}^2 e^{-f}\,d\mu_g = 1}\int_M \left\{4\abs{\nabla\psi}_g^2 + R_f\big(\abs{\psi}^2 - 1\big) \right\}e^{-f}\,d\mu_g.\]
Note that $\lambda_1(D)$ is independent of $f$, so by taking the maximizer over the following constraint on $f$ (with $g$ fixed), we get:
\[4\lambda_1(D)^2 - \lambda(g) = \sup_{\int_M e^{-f}\,d\mu_g = 1}\inf_{\int_M \abs{\psi}^2 e^{-f}\,d\mu_g = 1}\int_M \left\{4\abs{\nabla\psi}_g^2 + R_f\big(\abs{\psi}^2 - 1\big) \right\}e^{-f}\,d\mu_g. \]
where $\lambda(g)$ is Perelman's $\lambda$-functional:
\[\lambda(g) := \inf_{\int_M e^{-f}\,d\mu_g} \int_M R_f e^{-f}\,d\mu_g.\]
It has been shown by Perelman in \cite{Perelman} that the minimizer is achieved by a function $\hat{f}$ such that $R_{\hat{f}} = \lambda(g)$, and so we have:
\begin{align*}
4\lambda_1(D)^2 - \lambda(g) & = \inf_{\int_M\abs{\psi}^2 e^{-\hat{f}}d\mu_g = 1} \int_M \left\{4\abs{\nabla\psi}^2_g + \lambda(g) \big(\abs{\psi}^2-1\big)\right\} e^{-\hat{f}}\,d\mu_g\\
&  = \inf_{\int_M\abs{\psi}^2 e^{-\hat{f}}d\mu_g = 1}\int_M 4\abs{\nabla\psi}_g^2\,d\mu_g = -4\lambda_1(\Delta_f)
\end{align*}
where $\lambda_1(\Delta_f)$ is the lowest eigenvalue of $\Delta_f$ when acting on spinors.
This shows Perelman's $\lambda$-functional can be alternatively written the difference between the eigenvalues of two operators.
\[\lambda(g) = 4\lambda_1(D)^2 - 4\lambda_1(\Delta_f).\]
In \cite{BO23}, Baldauf-Ozuch introduced the following functional:
\[\mathcal{E}(g,f,\psi) := \int_M \left\{4\abs{\nabla\psi}_g^2 + R_f\big(\abs{\psi}^2 - 1\big) \right\}e^{-f}\,d\mu_g\]
which is the RHS of \eqref{eq:BO-function}. Applying this functional on asymptotically Euclidean manifolds with non-negative $R_f$, Baldauf-Ozuch established the monotonicity of the min-max of $\mathcal{E}$ and the constancy of the ADM mass along the Ricci flow on such a manifold.

Another important functional introduced by Perelman is the $\mathcal{W}$-functional, defined as
\begin{align*}
\mathcal{W}(g,f,\tau) & := \frac{1}{(4\pi\tau)^{n/2}}\int_M \left\{\tau \big(R+\abs{\nabla f}^2\big) + f - n\right\}e^{-f}\,d\mu_g\\
& = \frac{1}{(4\pi\tau)^{n/2}}\int_M (\tau R_f + f - n)e^{-f}\,d\mu_g.
\end{align*}
Under the constraint $\int_M e^{-f}\,d\mu_g = (4\pi\tau)^{n/2}$, the critical points of $\mathcal{W}$ are shrinking Ricci solitons which model finite-time singularities.

Inspired by previous discussions about Perelman's $\mathcal{F}$- and $\lambda$-functionals, and Baldauf-Ozuch's $\mathcal{E}$-functional, we introduce the following $\mathbb{W}_\lambda$-functional which unifies these known functionals.
\begin{definition}
On a closed Riemannian spin manifold $(M^n, g)$, we define the spinorial Perelman's entropy $\mathbb{W}_\lambda$ as follows:
\begin{equation}
\label{def:generalform}
\mathbb{W}_\lambda(g,f,\psi,\tau) := \frac{1}{(4\pi\tau)^{n/2}}\int_M \left\{4\abs{\nabla\psi}^2_g + R_f\big(\abs{\psi}^2-\tau\big) - \lambda(f-n)\right\}e^{-f}\,d\mu_g
\end{equation}
where $g$ is any Riemannian metric on $M$, $f : M \to \R$ is any smooth function, $\psi$ is a spinor field, $\tau > 0$ is a positive scalar function of $t$, and $\lambda \in \R$ is a fixed constant.
\end{definition}

\subsection{Relation with Dirichlet spinorial energy}
When $\lambda = 0$, $\hat{f}$ is a function such that $R_{\hat{f}} = C$ which is a constant, and $\psi_0$ is a spinor with
\[\int_M e^{-\hat{f}}\,d\mu_g = (4\pi\tau)^{n/2} \quad \text{ and } \quad \int_M \abs{\psi_0}^2 e^{-\hat{f}}\,d\mu_g = c(4\pi\tau)^{n/2},\]
then $\mathbb{W}_\lambda(g,\hat{f},\hat{\psi},\tau)$ is essentially the Dirichlet spinor energy:
\[\mathbb{W}_\lambda(g,\hat{f},\hat{\psi},\tau) = \frac{1}{(4\pi\tau)^{n/2}}\int_M 4\abs{\nabla\psi}_g^2 \,d\mu_g + C(c-\tau).\]

\subsection{Relation with Perelman's $\mathcal{W}$- and $\mathcal{F}$-functionals}
The $f$-weighted Weitzenb\"ock's formula
\begin{equation}
\label{eq:f-Weitzenbock}
D_f^2\psi = -\Delta_f \psi + \frac{1}{4}R_f\psi	
\end{equation}
implies that
\begin{equation}
\label{eq:f-Weitzenbock-with-norm}
2\Delta_f \abs{\psi}^2  =   R_f \abs{\psi}^2  + 4\abs{\nabla\psi}^2 - 4\textup{Re} \langle D_f^2 \psi ,\psi\rangle.	
\end{equation}

Since $M$ is closed, we have $\displaystyle{\int_M (\Delta_f \varphi)e^{-f}\,d\mu_g = 0}$ for any scalar function $\varphi$, and \eqref{def:generalform} becomes
\begin{align*}
\mathbb{W}_\lambda(g, f, \psi,\tau) & = \frac{1}{(4\pi\tau)^{n/2}}\int_M \left\{4\textup{Re} \langle D_f^2 \psi ,\psi\rangle + 2\Delta_f\abs{\psi}^2 - \tau R_f - \lambda (f-n)\right\}e^{-f}\,d\mu_g\\
& = \frac{1}{(4\pi\tau)^{n/2}}\int_M \left\{4\abs{D_f\psi}^2 - \tau \big(R+\abs{\nabla f}_g^2\big) - \lambda (f-n)\right\} e^{-f}\,d\mu_g.
\end{align*}
A spinor $\hat\psi \in \ker D_f$ is called a \emph{$f$-weighted harmonic spinor}. When we restrict $\mathbb{W}_\lambda$ on $f$-weighted harmonic spinors, we can recover Perelman's $\mathcal{W}$-functional or $\mathcal{F}$-functional, depending on the choice of $\lambda$:
\begin{align*}
\mathbb{W}_\lambda(g,f,\hat\psi,\tau) & = -\frac{1}{(4\pi\tau)^{n/2}}\int_M \left\{\tau\big(R + \abs{\nabla f}_g^2\big) + \lambda(f - n)\right\}e^{-f}\,d\mu_g\\
& = \begin{cases}
-\mathcal{W}(g,f,\tau) & \text{ if } \lambda = 1\\
-\frac{\tau}{(4\pi\tau)^{n/2}}\mathcal{F}(g,f) & \text{ if } \lambda = 0
\end{cases}.
\end{align*}

\subsection{Relation with Baldauf-Ozuch's functional}
When $\lambda = 0$, and $\tau \equiv 1$, \eqref{def:generalform} becomes:
\[\mathbb{W}_\lambda(g,f,\psi,1) = \frac{1}{(4\pi)^{n/2}}\int_M \left\{4\abs{\nabla\psi}_g^2 + R_f\big(\abs{\psi}^2-1\big)\right\}e^{-f}\,d\mu_g = \frac{1}{(4\pi)^{n/2}}\mathcal{E}_g(f,\psi)\]
which is (up to a constant multiple) the spinor energy functional studied in Baldauf-Ozuch's work \cite{BO23}.

\subsection{Relation with Perelman's $\mu$- and $\lambda$-functionals}
\label{subsect:Perelman}
For simplicity, we denote
\[d\Omega_f = \frac{1}{(4\pi\tau)^{n/2}}e^{-f}d\mu_g.\]
We then consider the constraint over all spinors $\psi$ with fixed $L^2(d\Omega_f)$-norm.
\[\inf_{\int_M \abs{\psi}^2 d\Omega_f = c} \mathbb{W}_\lambda(g,f,\psi,\tau) = \inf_{\int_M \abs{\psi}^2 d\Omega_f = c}\int_M 4\abs{D_f\psi}^2 d\Omega_f - \int_M \left\{\tau(R + \abs{\nabla f}^2_g) + \lambda(f-n)\right\}\,d\Omega_f,\]
where the infimum is taken with $(g,f,\tau)$ being fixed, and $c > 0$ is a constant. 

As discussed, we have
\[\inf_{\int_M \abs{\psi}^2 d\Omega_f = c}\int_M 4\abs{D_f\psi}^2 d\Omega_f = \int_{\int_M \abs{\psi}^2\,d\Omega_0 = c} \int_M 4\abs{D\psi}^2 d\Omega_0 = 4c\lambda_1(D)^2,\]
so by further taking the supremum over the constraint $\int_M d\Omega_f = 1$, we get
\[\sup_{\int_M d\Omega_f = 1} \inf_{\int_M \abs{\psi}^2 d\Omega_f = c} \mathbb{W}_\lambda(g,f,\psi,\tau) = 
\begin{cases}
4c\lambda_1(D)^2 - \mu(g,\tau) & \text{ if } \lambda = 1\\
4c\lambda_1(D)^2 - \tau(4\pi\tau)^{n/2}\lambda(g) & \text{ if } \lambda = 0
\end{cases}
\]
It was proved by Perelman in \cite{Perelman} that the minimizers $\hat{f}$ for the $\mu$- and $\lambda$-functionals exists, and the following holds:
\[\tau R_{\hat{f}} + \lambda(\hat{f}-n) =
\begin{cases}
\mu(g,\tau) & \text{ if } \lambda = 1\\
\tau\lambda(g) & \text{ if } \lambda = 0.
\end{cases}
\]
Therefore, we have
\begin{equation}
\label{eq:Friedrich}
\sup_{\int_M d\Omega_f = 1} \inf_{\int_M \abs{\psi}^2 d\Omega_f = c} \mathbb{W}_\lambda(g,f,\psi,\tau) = 
\begin{cases}
4c\lambda_1(D)^2 - \tau R_f - f + n & \text{ if } \lambda = 1\\
4c\lambda_1(D)^2 - \tau(4\pi\tau)^{n/2}R_f & \text{ if } \lambda = 0
\end{cases}	
\end{equation}
The well-known Friedrich's inequality \cite{Frd80} and its $f$-weighted generalization in \cite{BO22} said that on a closed spin manifold $(M^n,g)$ and $f \in C^\infty(M,\R)$, then
\[\lambda_1(D)^2 \geq \frac{n}{4(n-1)}\min R_f,\]
with equality holds if and only if $f$ is a constant and $(M^n,g)$ admits a Killing spinor, and hence $(M^n,g)$ is Einstein. From \eqref{eq:Friedrich} and the Friedrich's inequality, if one could smoothly deform the data $(g,f,\psi)$ with some suitable choice of constraints ($\tau, c$) and under a suitable flow such that $\mathbb{W}_0$ decreases to $0$ as $t \to \infty$, then the limit metric $g_\infty$ should be an Einstein metric. It would also be interesting to discover whether or not there is any generalization of Friedrich's inequality so that the equality holds exactly when $(M^n,g,f)$ is a gradient Ricci soliton. We will show in later section that the constraint critical point of $\mathbb{W}_1$ is achieved by ``twisted'' Ricci solitons.  The case $\lambda = 1$ in \eqref{eq:Friedrich} seems to suggest a possible form of generalized Friedrich's inequality with equality case achieved by Ricci solitons.

\subsection{Gradient flow}
Inspired by \eqref{eq:Friedrich} and the potential relations with Friedrich's inequality, we will introduce a flow system on $(g,f,\psi)$ which preserves the constraints
\[\int_M d\Omega_f = 1 \quad \text{ and } \quad \abs{\psi}^2 = c.\]
The flow system is of the form:
\begin{align*}
\D{g}{t} & = -2\left(\Ric  + \mathcal{L}_{\frac{1}{2}\nabla f  - \frac{1}{\tau}V_f}g - \frac{2}{\tau}S_{g,f,\psi} - \frac{\lambda}{2\tau}g\right)\\
\D{f}{t} & =  - \Delta f - R + \frac{\lambda n}{2\tau} +\frac{4}{\tau} \Tr_g S + \frac{2}{\tau}\div\,V_f \\
\D{\psi}{t} & =  \Delta_f \psi +  \frac{|\nabla\psi|^2}{|\psi|^2}\psi  
\end{align*}
where the vector field $V_f$ and the symmetric $2$-tensor $S_{g,f,\psi}$ will be defined in later sections. We will establish the short-time existence (Theorem \ref{thm;short-time-existence}) of this flow. The $\mathbb{W}_\lambda$-functional will be shown to be monotone decreasing along the flow (Proposition \ref{lma:monotonicity_A}), and it is stationary on ``twisted'' Ricci solitons. Therefore, in some sense, the flow system is the negative $L^2$-gradient flow of the $\mathbb{W}_\lambda$-function. 

To establish the short-time existence, we follow a DeTurck-trick type strategy adapted to the coupled
$(g,f,\psi)$-system. More precisely, in Section~4 we introduce a gauged system by adding the DeTurck
vector field $X = W(g,g_0) - \frac{2}{\tau}V_f + \nabla f$ to the metric equation and compensating the induced diffeomorphism action in
the $(f,\psi)$-equations by the corresponding transport terms (including the spinorial Kosmann Lie
derivative). A key analytic feature of this gauge is that it puts the system into a \emph{triangular
parabolic form} at the level of highest-order derivatives: the metric equation becomes strictly parabolic
with a Laplace-type principal part in $g$ and contains no second derivatives of $f$ or $\psi$; the spinor
equation is strictly parabolic in $\psi$, while the additional top-order terms involve at most second
derivatives of $g$ (coming from the Kosmann correction) but only first derivatives of $\psi$; and the scalar
equation is (forward/backward) parabolic in $f$ depending on whether $c>\tau$ or $c<\tau$.

In Section~4.1 we compute the principal symbols and verify that, under the pointwise normalization
$|\psi_0|^2\equiv c$, the $(g,\psi)$-subsystem is uniformly parabolic, whereas the $f$-equation is uniformly
forward parabolic if $c>\tau$ and uniformly backward parabolic if $c<\tau$. In Section~4.2 we exploit the
triangular structure to solve the gauged system by a Banach fixed point argument in parabolic H\"older
spaces: we solve successively for $g$, then $\psi$, and then $f$, applying parabolic Schauder estimates at
each step and using a short-time interpolation estimate to gain a small factor in the difference
bounds and obtain a contraction for sufficiently small time. Finally, we undo the gauge by pulling back
along the time-dependent diffeomorphisms generated by the gauge vector field, which yields a solution
to the original flow \eqref{eq:coupled_dg}--\eqref{eq:coupled_dpsi}.

The article is structured as follows. We will derive the first variation formulae of $\mathbb{W}_\lambda$ in Section 2 using some of the formulae derived in \cite{Perelman} and \cite{BO22,BO23}. In Section 3, we will derive the Euler-Lagrange's equations under some natural constraints that will inspire us to consider the spinorial Ricci flow. In Section 4, we will prove the short-time existence and uniqueness of the spinorial Ricci flow.

\medskip

\noindent\textbf{Acknowledgement.} A.Chow is partially supported by the Croucher Foundation Start-up Grant and the HKUST New Faculty Start-up Grant. 
F.Fong is partially supported by the General Research Fund \#16305625 by the Hong Kong Research Grants Council. The authors would like to thank Jingbo Wan for helpful discussions.

\section{Variation Formulae}
\subsection{Evolution equations}
In this section, we first list the evolution equations of various important quantities, and then apply them to find out the first variation of the $\mathbb{W}_\lambda$-functional.
Suppose $g(t)$, $f(t)$, $\tau(t)$, and $\psi(t)$ are $1$-parameter family of Riemannian metrics, scalar functions, and spinors on $M$. We denote $\dot{g}$, $\dot{f}$, $\dot\tau$, $\dot\psi$ their derivatives with respect to $t$. 
\begin{lemma}[c.f. \cite{AWW}, Propositions 2.6, 2.25 of \cite{BO23}] 
\begin{align}
\label{ddt_|grad psi|^2}
\D{}{t}\abs{\nabla\psi}_g^2 & = - \big\langle \dot{g}, \langle \nabla\psi \otimes \nabla\psi\rangle\big\rangle + \frac{1}{2}\langle \nabla\dot{g}, T_\psi\rangle + 2\textup{Re}\langle \nabla\dot\psi, \nabla\psi\rangle\\
\D{}{t}R_f & = \div_f \div_f \dot{g} - 2\Delta_f\left(\frac{1}{2}\tr_g\dot{g} - \dot{f}\right) - \langle \dot{g},\Ric_f\rangle\\
\D{}{t}\abs{\psi}^2 & = 2\textup{Re}\langle \psi,\dot\psi\rangle\\
\D{}{t}\frac{1}{(4\pi\tau)^{n/2}}(e^{-f}\,d\mu_g) & = \frac{1}{(4\pi\tau)^{n/2}}\left(\frac{1}{2}\tr_g \dot{g} - \dot{f} - \frac{n\dot\tau}{2\tau}\right)e^{-f}\,d\mu_g
\end{align}
where $T_\psi$ and $\langle \nabla\psi \otimes \nabla\psi\rangle$ are respectively the $3$- and $2$-tensors defined by:
\begin{align}
T_\psi(X,Y,Z) & = \frac{1}{2}\textup{Re}\left((X \wedge Y) \cdot \psi, \nabla_Z \psi\rangle + \langle (X \wedge Z) \cdot \psi, \nabla_Y \psi\rangle \right) \label{eq:T_psi}\\
\langle \nabla\psi \otimes \nabla\psi\rangle(X,Y) & = \textup{Re} \langle \nabla_X \psi, \nabla_Y \psi\rangle \label{eq:nabla_psi^2}
\end{align}
for any vector fields $X, Y, Z$.
\end{lemma}

Consequently, the integrals below satisfy the following evolution equations.
\begin{lemma}[c.f. \cite{AWW}, Propositions 2.20 and 2.25 of \cite{BO23}]
\label{lma:integral_variations}
Assume $\partial M = \emptyset$, we have
\begin{align}
\label{eq:ev_R_f} & \frac{d}{dt}\frac{1}{(4\pi\tau)^{n/2}}\int_M R_f e^{-f}\,d\mu_g\\
 & = \frac{1}{(4\pi\tau)^{n/2}}\int_M \left\{\left(\frac{1}{2}\tr_g(\dot{g}) - \dot{f} - \frac{n\dot\tau}{2\tau}\right)R_f -\big\langle \dot{g}, \Ric_f \big\rangle\right\} e^{-f}\,d\mu_g \nonumber\\
\label{eq:ev_R_f|psi|^2} & \frac{d}{dt}\frac{1}{(4\pi\tau)^{n/2}}\int_M R_f \abs{\psi}^2 e^{-f}\,d\mu_g\\
& = \frac{1}{(4\pi\tau)^{n/2}}\int_M \left\{-\big\langle \dot{g}, \Ric_f \abs{\psi}^2\big\rangle + \abs{\psi}^2 \div_f\div_f \dot{g} + 2\textup{Re}\langle \dot\psi, R_f\psi\rangle\right. \nonumber\\
& \hskip 3cm + \left.4\left(\frac{1}{2}\tr_g\dot{g} - \dot{f} - \frac{n\dot\tau}{2\tau}\right)\left(\textup{Re}\langle D_f^2\psi,\psi\rangle - \abs{\nabla\psi}^2_g\right)\right\} e^{-f}\,d\mu_g\nonumber\\
\label{eq:ev_|nabla_psi|^2} & \frac{d}{dt}\frac{1}{(4\pi\tau)^{n/2}}\int_M \abs{\nabla\psi}_g^2 e^{-f}\,d\mu_g \\
& = \frac{1}{(4\pi\tau)^{n/2}}\int_M \left\{\left(\frac{1}{2}\tr_g\dot{g} - \dot{f} - \frac{n\dot\tau}{2\tau}\right) \abs{\nabla\psi}^2 - 2\textup{Re} \langle\dot\psi, \Delta_f\psi\rangle\right. \nonumber\\
& \hskip 3cm - \left.\left\langle \dot{g}, \frac{1}{2}\div_f T_\psi + \langle \nabla\psi \otimes \nabla\psi\rangle\right\rangle \right\}e^{-f}\,d\mu_g\nonumber
\end{align}
\end{lemma}
Combining these, we can then compute the first variation formula for $\mathbb{W}_\lambda$.

\begin{proposition}
\label{prop:var_W}
Suppose $\partial M = \emptyset$. The first variation of $\mathbb{W}_\lambda$ is given by
\begin{align}
\label{eq:var_W_U}
		& \frac{d}{dt}\mathbb{W}_\lambda(g,f,\psi,\tau)\\
&=  \frac{1}{(4\pi\tau)^{n/2}}\int_M \left\{ \left\langle \tau\dot{g} - \dot{\tau}g,\, \Ric_f - \frac{\lambda}{2\tau}g  \right\rangle - \left\langle \dot{g},\,  \mathcal{L}_{V} g + 2S \right\rangle + 8\textup{Re} \left\langle D_f^2\psi,\, \dot{\psi} \right\rangle\right.\nonumber\\
&\hskip 2.8cm + \left.\left(\frac{1}{2}\tr_g(\dot{g}) - \dot{f} -  \frac{n\dot\tau}{2\tau}\right)\left(4\textup{Re}\langle D_f^2\psi, \psi\rangle-\tau R_f - \lambda(f-n-1)\right)\right\} e^{-f} d\mu\nonumber
	\end{align}
where $V$ is a vector field on $M$ satisfying
\begin{align}
\label{eq:V}
V_{g,f,\psi} & = \sum_{i=1}^n \textup{Re}\langle \psi, e_i \cdot D_f\psi\rangle e_i,
\end{align}
and $S$ is the symmetric $2$-tensor defined by
\begin{equation}
\label{eq:S}
S_{g,f,\psi}(X,Y) := \textup{Re}\langle D_f\psi, X \cdot \nabla_Y \psi + Y \cdot \nabla_X \psi\rangle.
\end{equation}
If $(g,f,\psi)$ are clear from the context, we will abbreviate $V_{g,f,\psi}$ and $S_{g,f,\psi}$ by simply $V$ and $S$ respectively.

\end{proposition}
\begin{proof}
First, it is useful to know
\begin{equation}
\int_M \abs{\psi}^2 (\div_f\div_f \dot{g}) e^{-f}\,d\mu_g = \int_M \langle \dot{g},\nabla\nabla\abs{\psi}^2 \rangle e^{-f}\,d\mu_g.
\end{equation}
We also need the following evolution equation which can be easily obtained by direct computations:
\begin{align}
\label{eq:ev_f} & \frac{d}{dt}\frac{1}{(4\pi\tau)^{n/2}}\int_M (f-n)e^{-f}\,d\mu_g\\
& = \frac{1}{(4\pi\tau)^{n/2}}\int_M \left\{\dot{f} + \left(\frac{1}{2}\tr_g \dot{g} - \dot{f} - \frac{n\dot\tau}{2\tau}\right)(f-n)\right\}e^{-f}\,d\mu_g. \nonumber
\end{align}
By adding up the above identities to the evolution equations in Lemma \ref{lma:integral_variations}, we obtain
\begin{align}
\label{eq:var_W_1}
& \frac{d}{dt}\mathbb{W}_\lambda(g,f,\psi,\tau) = 4\eqref{eq:ev_|nabla_psi|^2} + \eqref{eq:ev_R_f|psi|^2} - \tau\eqref{eq:ev_R_f}-\lambda\eqref{eq:ev_f} - \frac{\dot\tau}{(4\pi\tau)^{n/2}}\int_M R_f e^{-f}\,d\mu_g\\
& = \frac{1}{(4\pi\tau)^{n/2}}\int_M \left\{\left\langle \dot{g}, -2\div_f T_\psi -4\textup{Re}\langle\nabla\psi\otimes\nabla\psi\rangle  - \Ric_f \abs{\psi}^2 + \nabla\nabla\abs{\psi}^2 + \tau \Ric_f\right\rangle\right. \nonumber\\
& \hskip 3cm + \textup{Re}\langle \dot\psi, -8\Delta_f \psi+ 2R_f\psi\rangle - \dot\tau R_f - \lambda\dot{f}\nonumber\\
& \hskip 3cm + \left.\left(\frac{1}{2}\tr_g(\dot{g}) - \dot{f} - \frac{n\dot\tau}{2\tau}\right)\left(4\textup{Re}\langle D_f^2\psi,\psi\rangle - \tau R_f - \lambda(f-n)\right)\right\}e^{-f}\,d\mu_g\nonumber\\
& = \frac{1}{(4\pi\tau)^{n/2}} \int_M \left\{\left\langle \tau\dot{g}, \Ric_f - \frac{\lambda}{2\tau}g\right\rangle - \dot\tau \left(R_f - \frac{\lambda n}{2\tau}\right)  + 8\textup{Re} \langle \dot\psi, D_f^2\psi\rangle\right. \nonumber\\
& \hskip 3cm + \left\langle \dot{g}, -2\div_f T_\psi -4\textup{Re}\langle\nabla\psi\otimes\nabla\psi\rangle  - \Ric_f \abs{\psi}^2 + \nabla\nabla\abs{\psi}^2\right\rangle \nonumber\\
& \hskip 3cm + \left.\left(\frac{1}{2}\tr_g(\dot{g}) - \dot{f} - \frac{n\dot\tau}{2\tau}\right)\left(4\textup{Re}\langle D_f^2\psi,\psi\rangle - \tau R_f - \lambda(f-n-1)\right)\right\}e^{-f}\,d\mu_g\nonumber
\end{align}
By observing that
\[R_f - \frac{\lambda n}{2\tau} = \tr_g\left(\Ric_f - \frac{\lambda}{2\tau}g\right) + \Delta_f = \left\langle g, \Ric_f - \frac{\lambda}{2\tau}g\right\rangle + \Delta_f f,\]
we have proved:
\begin{align}
\label{eq:var_W}
& \frac{d}{dt}\mathbb{W}_\lambda(g,f,\psi,\tau)\\
&=  \frac{1}{(4\pi\tau)^{n/2}}\int_M \left\{ \left\langle \tau\dot{g} - \dot{\tau}g,\, \Ric_f - \frac{\lambda}{2\tau}g  \right\rangle \right.\nonumber\\
&\hskip 3cm + \left\langle \dot{g},\, -2\div_f T_\psi - 4\textup{Re}\langle\nabla\psi\otimes\nabla\psi\rangle  - \Ric_f \abs{\psi}^2 + \nabla\nabla\abs{\psi}^2  \right\rangle \nonumber\\
&\hskip 3cm -\left(\frac{1}{2}\tr_g(\dot{g}) - \dot{f} -  \frac{n\dot\tau}{2\tau}\right)\left(\tau R_f + \lambda(f-n-1)\right) \nonumber\\
&\hskip 3cm + \left.4\textup{Re} \left\langle D_f^2\psi,\, \left( \frac{1}{2}\tr_g(\dot{g}) - \dot{f} - \frac{n\dot\tau}{2\tau}\right)\psi + 2\dot{\psi} \right\rangle  \right\}  e^{-f} d\mu. \nonumber
\end{align}
Next we simplify the second term in the integrand. The following term in \eqref{eq:var_W}
\[\left\langle \dot{g},\, -2\div_f T_\psi - 4\textup{Re}\langle\nabla\psi\otimes\nabla\psi\rangle  - \Ric_f \abs{\psi}^2 + \nabla\nabla\abs{\psi}^2  \right\rangle\]
is closely related to $D_f\psi$ by the following identity (see (2.13) of \cite{BO23}):
\begin{align}
\label{eq:div_fT}
\div_fT_\psi (X,Y) & =	 -\frac{1}{2}\Ric_f(X,Y)\abs{\psi}^2 + \frac{1}{2}\nabla_X\nabla_Y\abs{\psi}^2 - 2\textup{Re}\langle \nabla_X \psi, \nabla_Y \psi\rangle\\
& \hskip 1.5cm +\textup{Re}\langle D_f\psi, X \cdot \nabla_Y \psi + Y \cdot \nabla_X \psi\rangle + \frac{1}{2}\mathcal{L}_V g \nonumber
\end{align}
where $V$ is the vector field on $M$ defined by \eqref{eq:V}.

Combining all results from \eqref{eq:var_W}, \eqref{eq:div_fT} and \eqref{eq:S}, we complete the proof of \eqref{eq:var_W_U}
\end{proof}

\subsection{More about $S$ and $V$}
We will explore more about the tensor $S$ and vector field $V$ under some special assumptions on $\psi$. The following observation is useful:
\begin{lemma}
\label{lma:realVF}
For any real vector field $X$, and a spinor $\psi$, we have
\[\textup{Re}\langle X \cdot \psi, \psi\rangle = 0\]
\end{lemma}
\begin{proof}
\begin{align*}
2\textup{Re} \langle X \cdot \psi, \psi \rangle & = \langle X \cdot \psi, \psi\rangle + \overline{\langle X \cdot \psi, \psi\rangle}\\
& = \langle X \cdot \psi, \psi\rangle + \langle \psi, X \cdot \psi\rangle\\
& = -\langle \psi, X \cdot \psi\rangle + \langle\psi, X \cdot \psi\rangle\\
& = 0.	
\end{align*}
	
\end{proof}

\begin{lemma}
If $\psi$ is an eigen-spinor of $D$ or $D_f$ with real eigenvalue $\beta$, then we have
\[\Tr_g S_{g,f,\psi} = 2\beta^2\abs{\psi}^2.\]	
\end{lemma}
\begin{proof}
It is useful to observe that the trace of the tensor $S$ is given by:
\[\Tr_g S_{g,f,\psi} = \sum_i S_{g,f,\psi}(e_i,e_i) =\textup{Re}\langle D_f\psi, e_i \cdot \nabla_{e_i} \psi + e_i \cdot \nabla_{e_i} \psi\rangle = \textup{Re}\langle D_f\psi,2D\psi\rangle.\]	
If $D\psi = \beta\psi$ for some $\beta \in \R$, then
\[\textup{Re}\langle D_f\psi,2D\psi\rangle = \textup{Re}\left\langle \beta\psi - \frac{1}{2}\nabla f \cdot \psi, 2\beta\psi \right\rangle = 2\beta^2\abs{\psi}^2.\]
Similarly if $D_f\psi = \beta\psi$.

\end{proof}

A spinor $\psi$ is said to be a Killing spinor if there exists a constant $\mu \in \R$ such that $\nabla_X \psi = \mu X \cdot \psi$ for any real vector field $X$. In particular, if $\mu = 0$, we call $\psi$ a parallel spinor. In case that $\psi$ is a Killing spinor, the $S$-tensor can be written in terms of the metric $g$:

\begin{lemma}
Suppose there exists a constant $\mu \in \R$ such that $\nabla_X \psi = \mu X \cdot \psi$ for any vector field $X$ (i.e. $\psi$ is a Killing spinor), then we have
\[S_{g,f,\psi} = 2\mu^2 n \abs{\psi}^2 g\]
\end{lemma}

\begin{proof}
By the given condition $\nabla_X \psi = \mu X \cdot \psi$, we have
\begin{align*}
S(X,Y) & = \textup{Re}\langle D_f\psi, X \cdot \nabla_Y \psi + Y \cdot \nabla_X \psi\rangle\\
& = \textup{Re}\left\langle \sum_k e_k \cdot \nabla_{e_k}\psi - \frac{1}{2}\nabla f \cdot \psi , \mu (X \cdot Y + Y \cdot X) \cdot \psi \right\rangle\\
& = \textup{Re}\left\langle \sum_k e_k \cdot \mu e_k \cdot \psi - \frac{1}{2}\nabla f \cdot \psi, -2\mu g(X,Y) \psi\right\rangle\\
& = \textup{Re}\left\langle -\mu n \psi - \frac{1}{2}\nabla f \cdot \psi, -2\mu g(X,Y) \psi\right\rangle
\end{align*}	
Since $\nabla f$ is a real vector field, by Lemma \ref{lma:realVF} we have $2\textup{Re}\langle \nabla f \cdot \psi, \psi \rangle = 0$. It concludes that
\[S(X,Y) = 2\mu^2n \abs{\psi}^2 g(X,Y).\]
\end{proof}

It is useful to observe that if $\hat\psi$ is an $f$-weighted harmonic spinor, i.e. $D_f\hat\psi = 0$, then according to \eqref{eq:V} we have $V = 0$.  More generally, the same result holds if $\psi$ is an eigen-spinor of the $f$-weighted Dirac operator $D_f$.
\begin{lemma}\label{lma:V_zero}
	If $\psi\in\mathcal{J}$ is an eigen-spinor of $D_f$ with eigenvalue $\beta$, then the vector field $V$ defined in \eqref{eq:V} vanishes:
	\[	V=\sum_{i=1}^n\textup{Re}\langle\psi,\, e_i\cdot D_f\psi\rangle e_i = 0.\]
\end{lemma}

\begin{proof}
Direct consequence of Lemma \ref{lma:realVF}.
\end{proof}

If $\psi$ is a Killing spinor, then we can show that $V$ is parallel to the gradient $\nabla f$.
\begin{lemma}
Suppose there exists a constant $\mu \in \R$ such that $\nabla_X \psi = \mu X \cdot \psi$ for any vector field $X$ (i.e. $\psi$ is a Killing spinor), then we have
\[V = \abs{\psi}^2 \nabla f.\]
\end{lemma}

\begin{proof}
\begin{align*}
e_i \cdot D_f\psi & = e_i \cdot \Big(\sum_k e_k \cdot \nabla_{e_k}\psi - \frac{1}{2}\nabla f \cdot \psi\Big)\\
& = \mu \sum_k e_i \cdot e_k \cdot e_k \cdot \psi - \frac{1}{2}\sum_k(\nabla_k f) e_i \cdot e_k \cdot \psi\\
& = -\mu n e_i \cdot \psi - \frac{1}{2}\sum_k (\nabla_k f)e_i \cdot e_k \cdot \psi
\end{align*}
Then, we have
\begin{align*}
& \textup{Re}\langle \psi, e_i \cdot D_f \psi\rangle\\
& = \underbrace{\textup{Re}\langle \psi, -\mu n e_i \cdot \psi\rangle}_{=0 \text{ by Lemma \ref{lma:realVF}}} -\frac{1}{2}\sum_k (\nabla_k f)\textup{Re}\langle \psi, e_i \cdot e_k \cdot \psi\rangle\\
& = - \frac{1}{2}\sum_k (\nabla_k f)\left(\langle \psi, e_i \cdot e_k \cdot \psi\rangle + \langle e_i \cdot e_k \cdot \psi, \psi\rangle\right)\\
& = - \frac{1}{2}\sum_k (\nabla_k f)\left(\langle \psi, e_i \cdot e_k \cdot \psi \rangle + \langle \psi, e_k \cdot e_i \cdot \psi\rangle\right)\\
& = -\frac{1}{2}\sum_k (\nabla_k f) \langle \psi, -2\delta_{ik}\psi\rangle = \abs{\psi}^2 \nabla_i f 
\end{align*}
This proves our result.
	
\end{proof}

\section{Euler-Lagrange's Equations and Spinorial Ricci Flow}
\label{sect:EL_GF}

\subsection{Euler-Lagrange's equations}
After deriving the first variation formula for the $\mathbb{W}_\lambda$-functional in the previous section, we next discuss the Euler-Lagrange's equations under some natural constraints discussed below.

For any Riemannian metric $g$, function $f$, and positive constant $\tau$, we denote
\[d\Omega_{g,f,\tau} := e^{-f}(4\pi\tau)^{-n/2}d\mu_g.\]
If $g,f,\tau$ are clear from the context, we will simply denote it by $d\Omega$.

Define the following spaces of functions and spinors:
\begin{align*}
\mathcal{I}_{g,\psi,\tau} & := \left\{f \in C^\infty(M,\R) : \int_M d\Omega_{g,f,\tau} = 1\right\}	\\
\mathcal{J}_{g,f,\tau,c} & := \left\{\psi \in \Gamma^\infty(\Sigma M) : \int_M \abs{\psi}^2 d\Omega_{g,f,\tau} = c\right\}
\end{align*}
If $g,f,\psi,\tau,c$ are clear from the context, we simply denote them by $\mathcal{I}$ and $\mathcal{J}$.

Then, we find the critical points of $\mathbb{W}_\lambda$ subject to the constraints $f \in \mathcal{I}_{g,\psi,\tau}$ and $\psi \in \mathcal{J}_{g,f,\tau,c}$. It is useful to observe that
\begin{equation}
\label{eq:dOmega/dt}
\D{}{t}d\Omega_{g,f,\tau} = \left(\frac{1}{2}\tr_g(\dot{g}) - \dot{f} - \frac{n\dot\tau}{2\tau}\right)d\Omega_{g,f,\tau}.
\end{equation}
Therefore, we have
\begin{align}
\label{eq:dI/dt}
\frac{d}{dt}\int_M d\Omega_{g,f,\tau} & = \int_M \left(\frac{1}{2}\tr_g(\dot{g}) - \dot{f} - \frac{n\dot\tau}{2\tau}\right)d\Omega_{g,f,\tau},\\
\label{eq:dJ/dt}
\frac{d}{dt}\int_M \abs{\psi}^2 d\Omega_{g,f,\tau} & = \int_M \left\{2\textup{Re}\langle \dot\psi, \psi\rangle + \left(\frac{1}{2}\tr_g(\dot{g}) - \dot{f} - \frac{n\dot\tau}{2\tau}\right)\abs{\psi}^2\right\}d\Omega_{g,f,\tau}.
\end{align}

Next we use the method of Lagrange's multipliers to find out the Euler-Lagrange's equation of $\mathbb{W}_\lambda$ subject to constraints $f \in \mathcal{I}$ and $\psi \in \mathcal{J}$.

From \eqref{eq:var_W_U}, \eqref{eq:dI/dt}, and \eqref{eq:dJ/dt}, we have
\begin{align*}
& \frac{d}{dt}\mathbb{W}_\lambda - \alpha\frac{d}{dt}\int_M\,d\Omega - \beta \frac{d}{dt}\int_M \abs{\psi}^2 d\Omega\\
& = \int_M \bigg\{ \left\langle \tau\dot{g} - \dot{\tau}g,\, \Ric_f - \frac{\lambda}{2\tau}g  \right\rangle - \left\langle \dot{g}, \mathcal{L}_{V} g + 2S\right\rangle + 8\textup{Re} \left\langle \dot\psi, D_f^2\psi - \frac{\beta}{4}\psi\right\rangle \nonumber\\
&\hskip 2cm + \left(\frac{1}{2}\tr_g(\dot{g}) - \dot{f} -  \frac{n\dot\tau}{2\tau}\right)\underbrace{\left(4\textup{Re}\langle D_f^2\psi, \psi\rangle - \beta \abs{\psi}^2 -\tau R_f - \lambda(f-n-1) - \alpha\right)}_{(*)}\bigg\} d\Omega. \nonumber	
\end{align*}
Therefore, assuming $\tau > 0$ is a constant (i.e. $\dot\tau = 0$), the Euler-Lagrange's equations of $\mathbb{W}_\lambda$ subject to constraints $f \in \mathcal{I}$ and $\psi \in \mathcal{J}$ are given by:
\begin{align*}
\tau\left(\Ric_f - \frac{\lambda}{2\tau}g\right) - \mathcal{L}_{V}g - 2S + \frac{(*)}{2}g& = 0\\
D_f^2\psi & = \frac{\beta}{4}\psi\\
\underbrace{\tau R_f + \lambda(f-n-1) + \alpha + \beta\abs{\psi}^2 - 4\textup{Re}\langle D_f^2\psi, \psi\rangle}_{=(*)} & = 0	
\end{align*}
After simplification, we get the following:

\begin{proposition}
\label{prop:EL_constraint}
Fix $\tau > 0$. The Euler-Lagrange's equations of $\mathbb{W}_\lambda(g,f,\psi,\tau)$ under the constraints $f \in \mathcal{I}$ and $\psi \in \mathcal{J}$ are given by:	
\begin{align}
\label{eq:EL_constraint}
\Ric + \mathcal{L}_{\frac{1}{2}\nabla f-\frac{1}{\tau}V}g -\frac{2}{\tau}S - \frac{\lambda}{2\tau}g  & = 0, & D_f^2\psi & = \frac{\beta}{4}\psi, & \tau R_f + \lambda(f-n) = \lambda - \alpha,
\end{align}
where $S$ and $V$ are tensor and vector fields defined in \eqref{eq:S} and \eqref{eq:V} respectively. Therefore, if $(g,f,\psi)$ is a critical point, then $g$ is a ``twisted'' Ricci soliton, and $\psi$ is an eigen-spinor of $D_f^2$.
\end{proposition}
\begin{remark}
Using the constraints $f \in \mathcal{I}$ and $\psi \in \mathcal{J}$, one can also show that
\begin{align*}
\beta & = \frac{4}{c}\int_M \abs{D_f\psi}^2 d\Omega, & \mathbb{W}_\lambda = \beta c - \lambda + \alpha.
\end{align*}
\end{remark}
\begin{remark}
We call the metric $g$ of the critical point to be a ``twisted'' Ricci soliton because of the presence of the $S$-tensor term in \eqref{eq:EL_constraint}.
\end{remark}

\begin{remark}[From the $L^2(d\Omega)$-constraint to the unit spinor bundle]
In Proposition~\ref{prop:EL_constraint} we imposed the global constraint $\psi\in J_{g,f,\tau,c}$, i.e.
\[
\int_M |\psi|^2\,d\Omega_{g,f,\tau}=c,
\]
which yields the eigen-spinor equation $D_f^2\psi=\frac{\beta}{4}\psi$.

For the evolution problem, it is analytically advantageous (and standard in the spinor-flow literature
\cite{AWW}) to work on the \emph{unit spinor bundle}, i.e.\ to impose the stronger pointwise constraint
$|\psi|^2\equiv c$. Besides removing the scaling direction in the spinor variable, this normalization
has two key analytic consequences for our coupled flow being introduced next.

\smallskip
\noindent\emph{(1) Projected Euler--Lagrange equation.}
If one repeats the first-variation argument for $\mathbb{W}_\lambda$ under the single normalization
$\int_M d\Omega_{g,f,\tau}=1$, but restricts to variations tangent to $\{|\psi|^2\equiv c\}$
(so that $\mathrm{Re}\langle\dot\psi,\psi\rangle=0$ pointwise), then the spinor Euler--Lagrange equation
is the orthogonal projection of $D_f^2\psi$ onto $\psi^\perp$. Using the weighted Weitzenb\"ock formula
\eqref{eq:f-Weitzenbock}, this is equivalent to
\[
\Delta_f\psi+\frac{|\nabla\psi|^2}{c}\,\psi=0,
\]
which agrees with the stationary equation for the $\psi$-component of the flow \eqref{eq:coupled_dpsi}.

\smallskip
\noindent\emph{(2) Uniform control of the top-order coefficients and the parabolicity regime for $f$.}
The pointwise constraint $|\psi|^2\equiv c$ also freezes the highest-order coefficients in the gauged
system of Section~4. In particular, writing
\[
V_f \;=\; U + \frac{c}{2}\nabla f,
\qquad
U := \sum_i \mathrm{Re}\langle\psi,\,e_i\!\cdot D\psi\rangle\,e_i,
\]
we obtain
\[
\frac{2}{\tau}\,\mathrm{div}(V_f)
= \frac{c}{\tau}\,\Delta f + \frac{2}{\tau}\,\mathrm{div}(U),
\]
so the principal second-order part of the $f$-equation becomes
\[
-\Delta f + \frac{2}{\tau}\mathrm{div}(V_f)
= -\Bigl(1-\frac{c}{\tau}\Bigr)\Delta f + \text{(lower order terms)}.
\]
Consequently, after DeTurck gauge-fixing, the scalar equation is uniformly \emph{backward} parabolic
if $c<\tau$ and uniformly \emph{forward} parabolic if $c>\tau$ (degenerating at $c=\tau$). In the
fully forward regime $c>\tau$, the gauged system becomes a forward quasilinear parabolic evolution
in all variables $(g,f,\psi)$, and the monotonicity identity of Proposition~\ref{lma:monotonicity_A} below shows that
$\mathbb{W}_\lambda$ is a genuine Lyapunov functional for this forward parabolic gradient-like flow. This
forward-parabolic viewpoint suggests the potential applicability of dynamical tools such as
{\L}ojasiewicz--Simon inequalities for stability and convergence analysis near critical points,
in contrast to the classical Perelman setting where the $f$-equation is backward.

\smallskip
In summary, the flow introduced below can be viewed as the negative $L^2$-gradient flow of $\mathbb{W}_\lambda$
on the constraint space
\[
\Bigl\{(g,f,\psi): \int_M d\Omega_{g,f,\tau}=1,\;\;|\psi|^2\equiv c\Bigr\}.
\]
\end{remark}

\subsection{Spinorial Ricci flow}
Motivated by the above discussion and by the spinor flow of 
Ammann--Wei\ss--Witt \cite{AWW} (see also He--Wang \cite{HW} for regularity results), we introduce 
the following coupled evolution system. The flow is designed to preserve the normalization 
$\int_M d\Omega_{g,f,\tau}=1$ and, provided $|\psi_0|^2\equiv c$ pointwise, to evolve inside the unit 
spinor bundle $|\psi(t)|^2\equiv c$ (hence also $\int_M|\psi(t)|^2\,d\Omega_{g,f,\tau}=c$). 
The functional $\mathbb{W}_\lambda(g,f,\psi,\tau)$ is monotone decreasing along the flow. Moreover, in the regime $c>\tau$ the DeTurck-gauged system is fully forward parabolic, so that
$\mathbb{W}_\lambda$ serves as a Lyapunov functional for a forward parabolic evolution, which may be useful
for future stability and convergence analysis.

\newpage

\begin{definition}[Spinorial Ricci Flow]
On a Riemannian spin manifold $(M^n,g_0)$, we consider the following system of equations, called the \emph{spinorial Ricci flow}, with initial conditions $(g_0,f_0,\psi_0)$ so that $\abs{\psi_0}^2  \equiv c$ and $\tau > 0$ is fixed:
\begin{align}
\label{eq:coupled_dg}\D{g}{t} & = -2\left(\Ric  + \mathcal{L}_{\frac{1}{2}\nabla f  - \frac{1}{\tau}V_f}g - \frac{2}{\tau}S_{g,f,\psi} - \frac{\lambda}{2\tau}g\right)\\
\label{eq:coupled_df}\D{f}{t} & =  - \Delta f - R + \frac{\lambda n}{2\tau} +\frac{4}{\tau} \Tr_g S + \frac{2}{\tau}\div\,V_f \\
\label{eq:coupled_dpsi}\D{\psi}{t} & =  \Delta_f \psi +  \frac{|\nabla\psi|^2}{|\psi|^2}\psi  
\end{align}
\end{definition}

\begin{proposition}
\label{lma:monotonicity_A}
As long as solution exists, we have
\begin{align}
	\label{eq:monotonicity_A}
& \frac{d}{dt}\mathbb{W}_\lambda(g,f,\psi,\tau)\\
= &  -\int_M \left(2\tau \abs{\Ric + \mathcal{L}_{\frac{1}{2}\nabla f  - \frac{1}{\tau}V_f}g - \frac{2}{\tau}S   - \frac{\lambda}{2\tau}g}^2 + 8 \abs{ \Delta_f\psi + \frac{|\nabla\psi|^2}{|\psi|^2}\psi}^2\right)d\Omega \leq 0. \nonumber
\end{align}
Therefore, $\mathbb{W}_\lambda$ is monotone decreasing, and it is stationary if and only if $(g,f,\psi)$ is a ``twisted'' Ricci soliton given by
\[ \left(1-\frac{c}{\tau}\right)\Ric_f - \frac{\lambda}{2\tau}g= \frac{2}{\tau}\div_f T_\psi + \frac{4}{\tau} \textup{Re}\langle\nabla\psi\otimes\nabla\psi\rangle\]
and $\psi$ satisfies $\Delta_f \psi + \frac{\abs{\nabla\psi}^2}{c}\psi = 0$.

\end{proposition}

\begin{proof}
	Using the identity $\Delta_f|\psi|^2 = 2\textup{Re}\langle\Delta_f\psi,\, \psi\rangle + 2|\nabla\psi|^2$ and the Weitzenb\"ock identity, we have
\begin{align}
\label{eq:var_D}
& 8\textup{Re}\langle D_f^2\psi,\, \dot{\psi}\rangle \\
&= -8 \textup{Re}\langle \Delta_f\psi,\, \dot{\psi}\rangle + R_f  \frac{\partial}{\partial t}|\psi|^2  \nonumber	\\
&=  -8 \textup{Re}\left\langle \Delta_f\psi + \frac{|\nabla\psi|^2}{|\psi|^2}\psi ,\, \dot{\psi}\right\rangle + \left(\frac{4|\nabla\psi|^2}{|\psi|^2}  + R_f \right)\frac{\partial}{\partial t}|\psi|^2.	\nonumber
\end{align}
Next, we observe that
\begin{align*}
	\frac{\partial}{\partial t}|\psi|^2 &= 2\textup{Re}\langle \dot\psi, \psi\rangle\\
	&= 2\textup{Re}\langle \Delta_f\psi, \psi\rangle + \frac{|\nabla\psi|^2}{|\psi|^2}2\textup{Re}\langle \psi, \psi\rangle\\
	&= \Delta_f|\psi|^2.
\end{align*}
Hence, if $\abs{\psi}^2 \equiv c$ at $t = 0$, then $|\psi|^2 \equiv c$ as long as the flow exists. Consequently,
\begin{equation}
\label{eq:Df^2=||^2}
8\textup{Re}\langle D_f^2\psi,\, \dot{\psi}\rangle = -8 \abs{ \Delta_f\psi + \frac{|\nabla\psi|^2}{|\psi|^2}\psi}^2.
\end{equation}
Combining with \eqref{eq:var_W_U}, we proved the monotonicity formula \eqref{eq:monotonicity_A}.

Next, from the monotonicity formula we see that $\mathbb{W}_\lambda$ is stationary if and only if $(g,f,\psi)$ satisfies
	\[	\Ric + \mathcal{L}_{\frac{1}{2}\nabla f  - \frac{1}{\tau}V_f}g - \frac{2}{\tau}S   - \frac{\lambda}{2\tau}g = 0,\quad\text{and}\quad \Delta_f\psi + \frac{|\nabla\psi|^2}{|\psi|^2}\psi = 0.\]
Plugging \eqref{eq:div_fT} with $|\psi|^2\equiv c$ into the first identity above, we get 
\[\left(1-\frac{c}{\tau}\right)\Ric_f - \frac{\lambda}{2\tau}g= \frac{2}{\tau}\div_f T_\psi + \frac{4}{\tau} \textup{Re}\langle\nabla\psi\otimes\nabla\psi\rangle.\]
\end{proof}

\medskip

%\begin{corollary}
%	Fix $\lambda$ in the definition of $\mathbb W_\lambda$.
%If $(g,f,\psi)$ is a critical point of $\mathbb W_\lambda$ and $(g,f)$ is a gradient Ricci soliton
%satisfying $\Ric_f=\frac{\lambda}{2\tau}g$, then $\lambda\leq 0$.
%\end{corollary}

%\begin{proof}
%At a critical point we have
%\[
%\Bigl(1-\frac{c}{\tau}\Bigr)\Ric_f-\frac{\lambda}{2\tau}g
%=
%\frac{2}{\tau}\div_f T_\psi+\frac{4}{\tau}\textup{Re}\langle\nabla\psi\otimes\nabla\psi\rangle.
%\]
%If $\Ric_f=\frac{\lambda}{2\tau}g$, then the left-hand side equals $-\frac{c}{\tau}\frac{\lambda}{2\tau}g$.
%Taking the trace and integrating against $d\Omega$, the weighted divergence term integrates to zero,
%so we obtain
%\[
%-\frac{c}{\tau}\frac{n\lambda}{2\tau}\int_M d\Omega
%=\frac{4}{\tau}\int_M |\nabla\psi|^2\,d\Omega \ge 0.
%\]
%Since $\int_M d\Omega=1$, this forces $\lambda\le 0$. In particular, if $\lambda>0$ (shrinker case),
%no such critical point exists. If equality holds, then $\lambda=0$ and $\nabla\psi\equiv 0$.
%\end{proof}
%
%\medskip

\begin{lemma}\label{lem:critical_point}
	Let  $(g,f,\psi)$ be a critical point of $\mathbb W_\lambda$. Then
	\begin{align*}
		&\int_M R_f d\Omega = \frac{n\lambda}{2\tau} + \frac{4}{\tau}\int_M |D_f\psi|^2 d\Omega\\[1ex]
		\text{and}\quad &\left(1-\frac{c}{\tau}\right)\int_M |D_f\psi|^2 d\Omega = \frac{nc\lambda}{8\tau} + \int_M |\nabla\psi|^2 d\Omega.
	\end{align*}
	\[	\]
\end{lemma}

\begin{proof}
	Taking trace on the first twisted Ricci soliton equation and note that $\tr(\Ric_f) = R + \Delta f = R_f - \Delta_ff$, and integrate, we get
	\begin{align*}
		\Bigl(1-\frac{c}{\tau}\Bigr)\int_M R_f d\Omega -\frac{n\lambda}{2\tau} = \frac{4}{\tau}\int_M |\nabla\psi|^2 d\Omega.
	\end{align*}
	Next, we have
	\begin{align*}
		\int_M |\nabla\psi|^2 d\Omega &= \int_M |D_f\psi|^2 d\Omega - \frac{c}{4}\int_M R_f d\Omega.
	\end{align*}
	This implies the identities.
\end{proof}

\begin{corollary}\label{cor:regimes}
Let $(g,f,\psi)$ be a critical point of $\mathbb{W}_\lambda$ on the constraint space
\[
\left\{(g,f,\psi):\int_M d\Omega_{g,f,\tau}=1,\;|\psi|^2\equiv c\right\},
\]
with $\tau>0$ fixed and $c>0$. Then the identities in Lemma~\ref{lem:critical_point} imply:

\begin{enumerate}
\item[(i)] \textbf{(Fully forward parabolic regime $c>\tau$.)}
One necessarily has $\lambda\le 0$. Moreover, if $\lambda=0$ then
\[
\nabla\psi\equiv 0
\qquad\text{and}\qquad
D_f\psi\equiv 0.
\]
In particular $\nabla f\equiv 0$, and hence $\Ric_f=\Ric$; plugging into the stationary equation
forces $\Ric\equiv 0$. Hence the only steady critical points in the fully forward regime are Ricci-flat with parallel spinor and constant $f$.

\item[(ii)] \textbf{(Degenerate regime $c=\tau$.)}
One necessarily has $\lambda\le 0$, and in fact
\[
\int_M |\nabla\psi|^2\,d\Omega \;=\; -\,\frac{n}{8}\,\lambda.
\]
In particular, if $\lambda=0$ then $\nabla\psi\equiv 0$.

\item[(iii)] \textbf{(Backward--forward regime $c<\tau$.)}
One has the identity
\[
\int_M |D_f\psi|^2\,d\Omega
=\frac{1}{1-\frac{c}{\tau}}
\left(\frac{nc\lambda}{8\tau}+\int_M |\nabla\psi|^2\,d\Omega\right),
\]
and therefore the necessary condition
\[
\frac{nc\lambda}{8\tau}+\int_M |\nabla\psi|^2\,d\Omega \;\ge\; 0.
\]
In particular, if $\lambda>0$ (shrinker sign) then
\[
\int_M |D_f\psi|^2\,d\Omega \;\ge\; \frac{nc\lambda}{8(\tau-c)} \;>\;0.\]
\end{enumerate}

In particular, \emph{shrinking} critical points ($\lambda>0$) can only occur in the backward--forward
regime $c<\tau$, and never in the fully forward regime $c>\tau$ (nor in the degenerate case $c=\tau$).
\end{corollary}

\begin{proof}
Let
\[
A:=\int_M |D_f\psi|^2\,d\Omega\ge 0,\qquad
B:=\int_M |\nabla\psi|^2\,d\Omega\ge 0.
\]
Lemma~\ref{lem:critical_point} gives
\[
\left(1-\frac{c}{\tau}\right)A=\frac{nc\lambda}{8\tau}+B.
\]
If $c>\tau$, then $1-\frac{c}{\tau}<0$ so the left-hand side is $\le 0$, hence the right-hand side
is $\le 0$, which forces $\lambda\le 0$. If moreover $\lambda=0$, then
$\left(1-\frac{c}{\tau}\right)A=B$ with negative coefficient, forcing $A=B=0$, i.e.
$\nabla\psi\equiv 0$ and $D_f\psi\equiv 0$. Since $\nabla\psi\equiv 0$ implies $D\psi\equiv 0$,
we have $0 =D_f\psi=D\psi-\frac12\nabla f\cdot\psi=-\frac12\nabla f\cdot\psi$, and because
$|\psi|^2\equiv c>0$ this implies $\nabla f\equiv 0$. The stationary metric equation then reduces
to $\Ric\equiv 0$ when $\lambda=0$.

If $c=\tau$, then the same identity reduces to $0=\frac{nc\lambda}{8\tau}+B$, giving $\lambda\le 0$
and $B=-\frac{nc\lambda}{8\tau}=-\frac{n}{8}\lambda$, hence $\nabla\psi\equiv 0$ when $\lambda=0$.

If $c<\tau$, then $1-\frac{c}{\tau}>0$ and solving for $A$ yields the stated formula and the
inequality $\frac{nc\lambda}{8\tau}+B\ge 0$. If $\lambda>0$, then $A\ge \frac{1}{1-c/\tau}\cdot
\frac{nc\lambda}{8\tau}=\frac{nc\lambda}{8(\tau-c)}$.
\end{proof}

\color{black}

\section{Short-time existence and uniqueness of Spinorial Ricci flow}

We next establish the short-time existence and uniqueness of the spinorial Ricci flow introduced in the previous section.

\begin{theorem}[Short-time existence and uniqueness]\label{thm;short-time-existence}
Let $(M,g_0)$ be a closed Riemannian spin manifold. Fix $\tau>0$, a smooth function
$f_0\colon M\to\mathbb{R}$, and a spinor $\psi_0\in\mathcal{S}$ of constant length
$|\psi_0|^2\equiv c$.
\begin{itemize}
\item[(i)] \emph{(Backward--forward parabolic regime.)}
If $c<\tau$, then there exists $\varepsilon=\varepsilon(g_0,f_0,\psi_0,\tau)>0$
such that the spinorial Ricci flow system \eqref{eq:coupled_dg}--\eqref{eq:coupled_dpsi}
admits a unique solution $(g(t),f(t),\psi(t))$ on $t\in[0,\varepsilon)$ satisfying
\[
g(0)=g_0,\qquad \psi(0)=\psi_0,\qquad f(\varepsilon)=f_0.
\]
In this case, after DeTurck gauge-fixing, the $(g,\psi)$-equations are uniformly
forward parabolic, while the $f$-equation is uniformly \emph{backward} parabolic.

\item[(ii)] \emph{(Fully forward parabolic regime.)}
If $c>\tau$, then there exists $\varepsilon=\varepsilon(g_0,f_0,\psi_0,\tau)>0$
such that the spinorial Ricci flow system \eqref{eq:coupled_dg}--\eqref{eq:coupled_dpsi}
admits a unique solution $(g(t),f(t),\psi(t))$ on $t\in[0,\varepsilon)$ satisfying
\[
g(0)=g_0,\qquad f(0)=f_0,\qquad \psi(0)=\psi_0.
\]
In this case, after DeTurck gauge-fixing, the system is uniformly \emph{forward
parabolic} in all variables $(g,f,\psi)$.
\end{itemize}
\end{theorem}

\begin{remark}
The distinction between the cases $c<\tau$ and $c>\tau$ reflects a genuine change
in the analytic character of the flow, corresponding respectively to a mixed
forward--backward system and a fully forward parabolic system.
\end{remark}

We will prove this theorem in the remainder of this section. 
Using $g_0$ as a fixed  background metric, we set the DeTurck vector field
	\[	W(g,g_0)^k = g^{ij}(\Gamma_{ij}^k(g) - \Gamma_{ij}^k(g_0)).\]
We note that provided $|\psi|$ is a constant, the vector field $V_f$ is given by $V_f^i = \text{Re}\langle\psi, e_i\cdot D\psi\rangle + \frac{1}{2}df(e_i)|\psi|^2$. Hence let us define a vector field $U$ by $U = \text{Re}\langle\psi, e_i\cdot D\psi\rangle e_i$.   Consider the following gauged-system:

\begin{align}
\label{eq:coupled_dg_gauged}\D{g}{t} & = -2\Ric + \mathcal{L}_{W(g,g_0)}g   + \frac{4}{\tau}S_{g,f,\psi} + \frac{\lambda}{\tau}g\\
\label{eq:coupled_du_gauged}\D{f}{t} & =   -\left( 1 - \frac{|\psi|^2}{\tau}\right)\Delta f - R + \frac{\lambda n}{2\tau} +\frac{4}{\tau} \Tr_g S_{g,f,\psi}\\
& \hskip 1cm + \frac{2}{\tau}\div\, U + (1-\frac{|\psi|^2}{\tau})|\nabla f|^2 + \langle \nabla f, W - \frac{2}{\tau}U\rangle \nonumber\\
\label{eq:coupled_dpsi_gauged}\D{\psi}{t} & =  \Delta_f \psi +\frac{|\nabla\psi|^2}{|\psi|^2}\psi  +(1-\frac{|\psi|^2}{\tau})\nabla_{\nabla f}\psi + \mathcal{L}_{W- \frac{2}{\tau}U}^{\text{spin}}\psi. 
\end{align}

Here $\mathcal{L}_X^{\text{spin}}$ denotes the spinorial (Kosmann) Lie derivative along $X$:
	\[	\mathcal{L}_X^{\text{spin}} \psi = \nabla_X\psi - \frac{1}{4}(\nabla_iX_j - \nabla_jX_i)e_i\cdot e_j\cdot\psi.\]
Observe that \eqref{eq:coupled_du_gauged} is forward parabolic if $c>\tau$; backward parabolic if $c<\tau$. Moreover, Observe that the modified flow still evolves inside the unit-spinor bundle. Indeed, if we start with $|\psi_0|^2 \equiv c>0$ point-wise on $M$, then by metric-compatibility we have 
\begin{equation}\label{eqn;spinor_norm_preserved_under_gauge}
	\begin{aligned}
			\frac{\partial}{\partial t}|\psi|^2 &= 2\textup{Re}\langle \dot\psi, \psi\rangle\\
	&= 2\textup{Re}\langle \Delta_f\psi, \psi\rangle + \frac{|\nabla\psi|^2}{|\psi|^2}2\textup{Re}\langle \psi, \psi\rangle + 2\textup{Re}\langle \mathcal{L}_{X}^{\text{spin}} \psi, \psi\rangle\\
	&= \Delta_f|\psi|^2 + X|\psi|^2,
	\end{aligned}
\end{equation}
hence $|\psi|^2 \equiv |\psi_0|^2$ as long as the flow exists.
\medskip

\subsection{Symbol calculation}

Before we show short-time existence, let us first compute the principal symbol of the system. Recall the definition of principal symbol of a linear operator $L$ of order $k$:
	\[	\sigma_{\xi}(L)u = \frac{i^k}{k!}L(f^k u)(x),\]
where $f\in C^{\infty}(M)$ with $f(x) = 0, df_x = \xi$.

Let us first recall that

\begin{lemma}\label{lem;Ricci;symbols}
	Let $(\dot{g}, \dot{\psi}) = (\eta, s)$. We have
		\begin{align*}
			&\sigma_{\xi}(D_{(g,\psi)}\Ric(g))_{ab}(\eta,s) = \frac{1}{2}|\xi|^2\eta(e_a,e_b) - \frac{1}{2}\Big(\xi_i\xi_a\,\eta(e_b,e_i) +  \xi_i\xi_b\,\eta(e_a,e_i) - \xi_a\xi_b\,\tr(\eta)\Big),\\[1ex]
			&\sigma_{\xi}(D_{(g,\psi)}R_g)(\eta,s) = |\xi|^2\tr(\eta) - \eta(\xi,\xi),\\[1ex]
			&\sigma_{\xi}(D_{(g,\psi)}(\mathcal{L}_{W(g,g_0)}g))_{ab}(\eta,s) = -\Big(\xi_i\xi_a\,\eta(e_b,e_i) +  \xi_i\xi_b\,\eta(e_a,e_i) - \xi_a\xi_b\,\tr(\eta)\Big).
		\end{align*}
\end{lemma}

\begin{lemma}\label{lem;symbol_2}
	Let $(\dot{g}, \dot{\psi}) = (\eta, s)$. We have
	\begin{align*}
		\sigma_{\xi}\left(  D_{(g,\psi)} \mathcal{L}_{U}^{\text{spin}}\psi \right)(\eta,s)
		&  = -\frac{1}{16}\sum_{i,a,b}\sum_{j\neq k}\, \xi_a\xi_j\, \eta(e_i, e_k)\, \textup{Re}\langle \psi,\, e_b\cdot e_i\cdot e_j\cdot e_k\cdot\psi\rangle \, e_a\cdot e_b\cdot \psi\\[1ex]
		&\quad  + \frac{1}{16}\sum_{i,a,b}\sum_{j\neq k}\, \xi_b\xi_j\, \eta (e_i, e_k)\, \textup{Re}\langle \psi,\, e_a\cdot e_i\cdot e_j\cdot e_k\cdot\psi\rangle \, e_a\cdot e_b\cdot\psi\\[1ex]
		&\quad -  \frac{1}{4}\sum_{a, b}\xi_b\, \textup{Re}\langle\psi,\, e_a\cdot \xi\cdot s\rangle \, e_a\cdot e_b\cdot \psi +  \frac{1}{4}\sum_{a, b} \xi_a\, \textup{Re}\langle\psi,\, e_b\cdot \xi\cdot s\rangle  \, e_a\cdot e_b\cdot \psi,\\[1ex]
		\sigma_{\xi}\left( D_{(g,\psi)} \mathcal{L}_{W(g,g_0) }^{\text{spin}}\psi \right)(\eta,s) &= -\frac{1}{4} \sum_k\sum_{i\neq j}(\xi_i\xi_k\,\eta(e_j, e_k) - \xi_j\xi_k\eta(e_i, e_k))e_i\cdot e_j\cdot \psi,\\[1ex]
	\sigma_{\xi}\left( D_{(g,\psi)}\Delta_g\psi \right)(\eta,s) &= -\frac{1}{4}\sum_i\sum_{j\neq k}\xi_i\xi_j\, \eta(e_i,e_k)e_j\cdot e_k\cdot\psi - |\xi|^2 s,\\[1ex]
	\sigma_{\xi}\left( D_{(g,\psi)}\div_g\, U \right)(\eta,s) &=  \frac{1}{4}\sum_i\sum_{j\neq k}\xi_i\xi_j\, \eta(e_i,e_k)\, \text{Re}\langle\psi,\,e_j\cdot e_k\cdot\psi\rangle + |\xi|^2\,\text{Re}\langle\psi, s\rangle\\
	&\quad  + \frac{1}{4}|\xi|^2\tr(\eta)\,|\psi|^2 -  \frac{1}{4}\eta(\xi,\xi)\,|\psi|^2.
	\end{align*}

\end{lemma}

\begin{proof}
	Since we wish to compute the principal symbol at the end, we only need to worry about the linearization at top order.	We have
	\begin{align}
		\left(D_{(g,\psi)}\mathcal{L}_{U}^{\text{spin}}\psi\right)(\eta,s) & = \mathcal{L}_{(D_g U)(\eta)}^{\text{spin}}\psi  + \mathcal{L}_{(D_\psi U)(s)}^{\text{spin}}\psi + \lot
	\end{align}
	
	Note that $U^a = \text{Re}\langle\psi, e_a\cdot D\psi\rangle $. 
	Using [AWW, Lemma 4.12], we calculate, up to highest order, 
	\begin{align}
		(D_g U^a)(\eta) &= \frac{1}{4}\sum_{i}\sum_{j\neq k}\, (\nabla_{e_j}\eta)(e_i, e_k)\, \textup{Re}\langle \psi,\, e_a\cdot e_i\cdot e_j\cdot e_k\cdot\psi\rangle  + \lot
	\end{align} 
	This gives
	\begin{equation}
		\begin{aligned}
			\mathcal{L}_{(D_g U)(\eta)}^{\text{spin}}\psi &=  -\frac{1}{4}\sum_{a\neq b}(\nabla_a (D_g U^b(\eta)) - \nabla_b (D_g U^a(\eta)))\, e_a\cdot e_b\cdot \psi + \lot\\
		&= -\frac{1}{16}\sum_{i,a,b}\sum_{j\neq k}\, (\nabla_{e_a}\nabla_{e_j}\eta)(e_i, e_k)\, \textup{Re}\langle \psi,\, e_b\cdot e_i\cdot e_j\cdot e_k\cdot\psi\rangle\, e_a\cdot e_b\cdot \psi\\
		&\quad + \frac{1}{16}\sum_{i,a,b}\sum_{j\neq k}\, (\nabla_{e_b}\nabla_{e_j}\eta)(e_i, e_k)\, \textup{Re}\langle \psi,\, e_a\cdot e_i\cdot e_j\cdot e_k\cdot\psi\rangle\, e_a\cdot e_b\cdot \psi + \lot
		\end{aligned}
	\end{equation}
For the another other top order terms\, we also calculate, up to highest order,
	\begin{align}
		&(D_\psi U^a)(s) = \textup{Re}\langle\psi,\, e_a\cdot Ds\rangle + \lot
	\end{align}
	This gives
	
	\begin{equation}
		\begin{aligned}
			\mathcal{L}_{(D_{\psi} U)(s)}^{\text{spin}}\psi &=  -\frac{1}{4}\sum_{a\neq b}(\nabla_a (D_{\psi} U^b(s)) - \nabla_b (D_{\psi}  U^a(s)))\, e_a\cdot e_b\cdot \psi + \lot\\
			&= -\frac{1}{4}\sum_{a\neq b}\textup{Re}\langle\psi,\, e_a\cdot \nabla_{e_b}Ds\rangle \, e_a\cdot e_b\cdot \psi  + \frac{1}{4}\sum_{a\neq b}\textup{Re}\langle\psi,\, e_b\cdot \nabla_{e_a}Ds\rangle \, e_a\cdot e_b\cdot \psi + \lot\\
		\end{aligned}
	\end{equation}
Consequently,

	\begin{align}
		&\sigma_{\xi}\left(D_{(g,\psi)}\mathcal{L}_{U}^{\text{spin}}\psi\right)(\eta,s)\\
		&= -\frac{1}{16}\sum_{i,a,b}\sum_{j\neq k}\, \xi_a\xi_j\, \eta(e_i, e_k)\, \textup{Re}\langle \psi,\, e_b\cdot e_i\cdot e_j\cdot e_k\cdot\psi\rangle \, e_a\cdot e_b\cdot \psi\\[1ex]
		&\quad  + \frac{1}{16}\sum_{i,a,b}\sum_{j\neq k}\, \xi_b\xi_j\, \eta (e_i, e_k)\, \textup{Re}\langle \psi,\, e_a\cdot e_i\cdot e_j\cdot e_k\cdot\psi\rangle \, e_a\cdot e_b\cdot\psi\\[1ex]
		&\quad -  \frac{1}{4}\sum_{a, b}\xi_b\, \textup{Re}\langle\psi,\, e_a\cdot \xi\cdot s\rangle \, e_a\cdot e_b\cdot \psi +  \frac{1}{4}\sum_{a, b} \xi_a\, \textup{Re}\langle\psi,\, e_b\cdot \xi\cdot s\rangle  \, e_a\cdot e_b\cdot \psi.
	\end{align}

Secondly,
\begin{align*}
	D_{(g,\psi)}(\mathcal{L}_{W(g,g_0)}^{\text{spin}}\psi)(\eta,s) =  -\frac{1}{4}\sum_{i\neq j}(\nabla_i (D_g W^j(\eta)) - \nabla_j (D_g W^i(\eta)))\, e_i\cdot e_j\cdot \psi + \lot,
\end{align*}
This gives
\begin{align*}
	\sigma_{\xi}(D_{(g,\psi)}(\mathcal{L}_{W(g,g_0) }^{\text{spin}}\psi))(\eta,s) &= -\frac{1}{4}\sum_k\sum_{i\neq j}(\xi_i\xi_k\,\eta(e_j, e_k) - \xi_j\xi_k\eta(e_i, e_k))e_i\cdot e_j\cdot \psi.
\end{align*}

Next, using [AWW, Lemma 4.12] again, we have
\begin{align*}
	D_{(g,s)}(\Delta_g\psi)(\eta,s) &= \frac{1}{4}\sum_i\sum_{j\neq k}(\nabla_{e_i}\nabla_{e_j}\eta)(e_i,e_k)e_j\cdot e_k\cdot\psi + \Delta_g s + \lot
\end{align*}
This gives
\begin{align*}
	\sigma_{\xi}(D_{(g,s)}(\Delta_g\psi))(\eta,s) &= -\frac{1}{4}\sum_i\sum_{j\neq k}\xi_i\xi_j\, \eta(e_i,e_k)e_j\cdot e_k\cdot\psi - |\xi|^2 s.
\end{align*}

Lastly, note that 
\begin{align*}
	\div_g\, U = \text{Re}\langle\psi, D^2\psi\rangle= -\text{Re}\langle\psi, \Delta\psi\rangle + \frac{1}{4}R \,|\psi|^2
\end{align*}
Therefore, using the symbols for $\Delta_g\psi$ and $R_g$, we obtain
\begin{align*}
	\sigma_{\xi}(D_{(g,s)}(\div_g\, U))(\eta,s) &=  \frac{1}{4}\sum_i\sum_{j\neq k}\xi_i\xi_j\, \eta(e_i,e_k)\, \text{Re}\langle\psi,\,e_j\cdot e_k\cdot\psi\rangle + |\xi|^2\,\text{Re}\langle\psi, s\rangle\\
	&\quad + \frac{1}{4}|\xi|^2\tr(\eta)\,|\psi|^2 -  \frac{1}{4}\eta(\xi,\xi)\,|\psi|^2.
\end{align*}
\end{proof}

\bigskip

\begin{lemma}\label{lem;symbol_3}
	Let $(\dot{g}, \dot{\psi}) = (\eta, s)$ with $\textup{Re}\langle \psi,s\rangle =0$. We have
	
	\begin{equation}
	\begin{aligned}
	\textup{Re}\left\langle \sigma_{\xi}\left(  D_{(g,\psi)} \mathcal{L}_{U}^{\text{spin}}\psi \right)(\eta,s),\,s \right\rangle & = -\frac{1}{8} |\psi|^2\,  \textup{Re}\langle \psi,\, c(\eta(\xi,\cdot)\wedge \xi)\, s\rangle + \frac{1}{2}  \sum_a\textup{Re} \langle\psi,\, e_a\cdot \xi\cdot s\rangle^2,\\[1ex]
	\textup{Re}\left\langle\sigma_{\xi}\left( D_{(g,\psi)} \mathcal{L}_{W(g,g_0)}^{\text{spin}}\psi \right)(\eta,s),\, s \right\rangle  &= -\frac{1}{2}\,\textup{Re}\langle  \psi,\,c(\eta(\xi,\cdot)\wedge\xi) s\rangle.\\[1ex]
	\textup{Re}\left\langle \sigma_{\xi}\left( D_{(g,\psi)}\Delta_g\psi \right)(\eta,s),\, s \right\rangle &= -\frac{1}{4}\textup{Re}\langle \psi,\, c( \eta(\xi,\cdot)\wedge\xi) s\rangle - |\xi|^2 |s|^2,\\[1ex]
	\langle \sigma_{\xi}\left( D_{(g,\psi)}\div_g\, U \right)(\eta,s),\, h\rangle  &=   \frac{1}{4}|\xi|^2\tr(\eta)\, h\,|\psi|^2 -  \frac{1}{4}\eta(\xi,\xi)\,h\,|\psi|^2.
		\end{aligned}
	\end{equation}

\end{lemma}
\begin{proof}
	Let us observe that	
	\begin{align*}
		e_i\cdot e_k &= \frac{1}{2}(e_i\cdot e_k + e_k\cdot e_i) + \frac{1}{2}(e_i\cdot e_k - e_k\cdot e_i)\\
		&= -\delta_{ik} + \frac{1}{2}(e_i\cdot e_k - e_k\cdot e_i).
	\end{align*}
	Using the fact that $\eta$ is symmetric, we get
	
	\begin{align*}
		 &-\sum_{i}\sum_{j\neq k}\, \xi_a\xi_j\, \eta(e_i, e_k)\, \textup{Re}\langle \psi,\, e_b\cdot e_i\cdot e_j\cdot e_k\cdot\psi\rangle\\
		 &= \sum_{i}\sum_{j\neq k}\, \xi_a\xi_j\, \eta(e_i, e_k)\, \textup{Re}\langle \psi,\, e_b\cdot e_i\cdot e_k\cdot e_j\cdot\psi\rangle\\
		 &= \sum_{i,j,k}\, \xi_a\xi_j\, \eta(e_i, e_k)\, \textup{Re}\langle \psi,\, e_b\cdot e_i\cdot e_k\cdot e_j\cdot\psi\rangle - \sum_{i,j}\, \xi_a\xi_j\, \eta(e_i, e_j)\, \textup{Re}\langle \psi,\, e_b\cdot e_i\cdot e_j\cdot e_j\cdot\psi\rangle\\
		 &= -\sum_{j}\, \xi_a\xi_j\,\tr(\eta)\, \textup{Re}\langle \psi,\, e_b\cdot  e_j\cdot\psi\rangle + \xi_a \, \textup{Re}\langle \psi,\, e_b\cdot  \eta(\xi,\cdot)\cdot\psi\rangle\\
		 &= \xi_a\xi_b\,\tr(\eta)|\psi|^2 - \xi_a\,\eta(\xi, e_b) |\psi|^2.
	\end{align*}
	The last line follows from the property
	
		\[\textup{Re}\langle \psi,\, e_b\cdot  e_j\cdot\psi\rangle = -\delta_{jb}|\psi|^2.	\]
	This gives	
	
	\begin{equation}\label{eqn1;symbol_simplification}
		\begin{aligned}
			&-\sum_{i,a,b}\sum_{j\neq k}\, \xi_a\xi_j\, \eta(e_i, e_k)\, \textup{Re}\langle \psi,\, e_b\cdot e_i\cdot e_j\cdot e_k\cdot\psi\rangle \, e_a\cdot e_b\cdot \psi\\[1ex]
		&=\sum_{a, b} \Big(\xi_a\xi_b\,\tr(\eta)|\psi|^2 - \xi_a\,\eta(\xi, e_b) |\psi|^2\Big)\, e_a\cdot e_b\cdot\psi\\[1ex]
		&= -|\psi|^2|\xi|^2\tr(\eta)\,\psi - |\psi|^2\xi\cdot\eta(\xi,\cdot)\cdot\psi.		
		\end{aligned}
	\end{equation}
	Next, we have
	
	\begin{equation}\label{eqn2;symbol_simplification}
		\begin{aligned}
		& - \xi_b\, \textup{Re}\langle\psi,\, e_a\cdot \xi\cdot s\rangle \, e_a\cdot e_b\cdot \psi + \xi_a\, \textup{Re}\langle\psi,\, e_b\cdot \xi\cdot s\rangle \, e_a\cdot e_b\cdot \psi\\[1ex]
		&= -  \textup{Re}\langle\psi,\, e_a\cdot \xi\cdot s\rangle \, e_a\cdot \xi\cdot \psi +  \textup{Re}\langle\psi,\, e_b\cdot \xi\cdot s\rangle \, \xi\cdot e_b\cdot \psi \\[1ex]
		&=  2 \textup{Re}\langle\psi,\,  e_a\cdot  \xi\cdot s\rangle \, \xi\cdot e_a\cdot \psi + 2 \textup{Re}\langle\psi,\, \xi\cdot \xi\cdot s\rangle \psi\\[1ex]
		&= 2 \textup{Re}\langle\psi,\,  e_a\cdot  \xi\cdot s\rangle \, \xi\cdot e_a\cdot \psi.
		\end{aligned}
	\end{equation}
	Here the last line follows from $\text{Re}\langle\psi, s\rangle = 0$.
	Putting (\ref{eqn1;symbol_simplification}) and (\ref{eqn2;symbol_simplification}) together with the identity $c(v\wedge w) = \frac{1}{2}(v\cdot w-w\cdot v)$ gives
	
\begin{equation}\label{eqn3;symbol_simplification}
	\begin{aligned}
		\sigma_{\xi}\left(D_{(g,\psi)}\mathcal{L}_{U}^{\text{spin}}\psi\right)(\eta,s)=  \frac{1}{8}|\psi|^2\, c(\eta(\xi,\cdot)\wedge \xi)\,\psi +\frac{1}{2}  \textup{Re}\langle\psi,\,  e_a\cdot  \xi\cdot s\rangle \, \xi\cdot e_a\cdot \psi.
	\end{aligned}
\end{equation}
 This implies the first identity. For the second identity, it follows from 
 	\begin{align*}
		 &- \sum_k\sum_{i\neq j}(\xi_i\xi_k\,\eta(e_j, e_k) - \xi_j\xi_k\eta(e_i, e_k))\,\text{Re}\langle e_i\cdot e_j\cdot \psi,\,s\rangle\\[1ex]
		 &= -\text{Re}\langle\xi\cdot \eta(\xi,\cdot)\cdot\psi,\,s\rangle + \text{Re}\langle\eta(\xi,\cdot)\cdot\xi\cdot\psi,\,s\rangle \\[1ex]
		 &= -\text{Re}\langle c(\xi\wedge \eta(\xi,\cdot))\psi,\,s\rangle + \text{Re}\langle c(\eta(\xi,\cdot)\wedge\xi) \psi,\,s\rangle\\[1ex]
		 &= -2\text{Re}\langle  \psi,\,c(\eta(\xi,\cdot)\wedge\xi) s\rangle.
	\end{align*}
	For the third identity, we have
	
	\begin{align*}
		&-\sum_i\sum_{j\neq k}\xi_i\xi_j\, \eta(e_i,e_k)\, \text{Re}\langle e_j\cdot e_k\cdot\psi,\, s\rangle \\[1ex]
		&= -\sum_{i,j,k}\xi_i\xi_j\, \eta(e_i,e_k)\, \text{Re}\langle e_j\cdot e_k\cdot\psi,\, s\rangle + \sum_{i,j}\xi_i\xi_j\, \eta(e_i,e_j)\, \text{Re}\langle e_j\cdot e_j\cdot\psi,\, s\rangle\\[1ex]
		&= -\text{Re}\langle\xi\cdot\eta(\xi,\cdot)\cdot\psi,\, s\rangle -\eta(\xi,\xi) \text{Re}\langle\psi,\, s\rangle \\[1ex]
		&= -\text{Re}\langle c(\xi\wedge \eta(\xi,\cdot))\psi,\, s\rangle\\[1ex]
		&= -\text{Re}\langle \psi,\, c( \eta(\xi,\cdot)\wedge\xi) s\rangle.
	\end{align*}	
	This gives the third identity. Lastly, we similarly deduce that
	
	\begin{align*}
		&\sum_i\sum_{j\neq k}\xi_i\xi_j\, \eta(e_i,e_k)\, \text{Re}\langle\psi,\,e_j\cdot e_k\cdot\psi\rangle\\[1ex]
		&= \sum_{i,j,k}\xi_i\xi_j\, \eta(e_i,e_k)\, \text{Re}\langle\psi,\,e_j\cdot e_k\cdot\psi\rangle - \sum_{i,j}\xi_i\xi_j\, \eta(e_i,e_j)\, \text{Re}\langle\psi,\,e_j\cdot e_j\cdot\psi\rangle\\[1ex]
		&= \text{Re}\langle \psi,\, \xi\cdot \eta(\xi,\cdot)\cdot \psi\rangle + \eta(\xi,\xi)\,\text{Re}\langle\psi,\,\psi\rangle\\[1ex]
		&= \text{Re}\langle \psi,\, c(\xi\wedge \eta(\xi,\cdot)) \psi\rangle\\[1ex]
		&=0.
	\end{align*}
	This gives the last identity. 
\end{proof}
\vspace{0.5cm}
Now, for each $\psi\in\mathcal{S}$ let us define  $A(\psi)^{kl}\in\text{End}(\mathcal{S})$  by
	\begin{align}
		A(\psi)^{kl}(s) := (g^{kl}\text{id})\, s + \frac{1}{\tau}\textup{Re}\langle\psi, e_a\cdot e_l\cdot s\rangle e_a\cdot e_k\cdot \psi.
	\end{align}
 Let us recall that 
	\begin{align*}
		- 2\Ric(g) + \mathcal{L}_{W(g,g_0)}g = g^{kl}\hat{\nabla}_k\hat{\nabla}_l g + \mathcal{Q}(g, \hat{\nabla} g),
	\end{align*}
	where $\hat{\nabla}$ is covariant derivative with respect to the fixed background metric $g_0$, and $\mathcal{Q}(g, \hat{\nabla} g)$ is a quadratic term depends only up to first order derivatives of $g$. We define
	\begin{equation}\label{def:lower_order_terms}
		\begin{aligned}
			\mathcal{F}(g,f,\psi) &:= -2\Ric(g) + \mathcal{L}_{W(g,g_0)}g   + \frac{4}{\tau}S_{g,f,\psi} + \frac{\lambda}{4\tau}g -  g^{kl}\hat{\nabla}_k\hat{\nabla}_l g\\
			&= \mathcal{Q}(g, \hat{\nabla} g) + \frac{4}{\tau}S_{g,f,\psi} + \frac{\lambda}{4\tau}g.\\[1ex]
			\mathcal{G}(g,f,\psi) &:= - R_g + \frac{\lambda n}{2\tau} +\frac{4}{\tau} \Tr_g S_{g,f,\psi} + \frac{2}{\tau}\div\, U + (1-\frac{|\psi|^2}{\tau})|\nabla f|^2 + \langle \nabla f, W - \frac{2}{\tau}U\rangle \\[1ex]
			\mathcal{H}(g,f,\psi) &:= \Delta_f \psi +\frac{|\nabla\psi|^2}{|\psi|^2}\psi  +(1-\frac{|\psi|^2}{\tau})\nabla_{\nabla f}\psi + \mathcal{L}_{W(g,g_0)- \frac{2}{\tau}U}^{\text{spin}}\psi -  A(\psi)^{kl}\hat{\nabla}_k\hat{\nabla}_l\psi
		\end{aligned}
	\end{equation}
	
\medskip
\begin{lemma}\label{lem:lower_order_terms}
	For every $(g,f,\psi)\in\Gamma(\text{Sym}^2(T^*M)\oplus\R \oplus\mathcal{S})$, we have
	\begin{equation}
		\begin{aligned}
			\sigma_{\xi}(D_{(g,f,\psi)}\mathcal{F}(g,f,\psi))(\eta,h,  s) &= 0,\\[1ex]
			\sigma_{\xi}(D_{(g,f,\psi)}\mathcal{H}(g,f,\psi))(\eta,h, s) &=  \left(\frac{3}{4}  - \frac{|\psi_0|^2}{4\tau}\right)\, c(\eta(\xi,\cdot)\wedge\xi) \psi.
		\end{aligned}
	\end{equation}
	In other words, $\mathcal{F}(g,f,\psi)$ contains no top order derivatives of $g,f,\psi$, whereas $\mathcal{H}(g,f,\psi)$ contains no top order derivatives of $f,\psi$.
\end{lemma}

\begin{proof}
	Let us abbreviate the equations (\ref{eq:coupled_dg_gauged}) and (\ref{eq:coupled_dpsi_gauged}) as
\begin{equation}\label{gauged;system}
	\begin{aligned}
		&\frac{\partial}{\partial t}g =  Q_g(g,f,\psi)\\[1ex]
		&\frac{\partial}{\partial t}\psi =  Q_{\psi}(g,f,\psi).
	\end{aligned}
\end{equation}
Since $S(g,f,\psi)$ is of first-order, we have
\begin{align*}
	\langle\sigma_{\xi}(D_{(g,f,\psi)}Q_g)(\eta,h,  s),\,\eta\rangle
	&=	\langle \sigma_{\xi}(D(-2\Ric(g) + \mathcal{L}_{\gamma W(g,g_0)}g))(\eta, h, s),\,\eta\rangle = -|\xi|^2 |\eta|^2.
	\end{align*}
This implies that $\sigma_{\xi}(D_{(g,f,\psi)}\mathcal{F}(g,f,\psi))(\eta,h,  s) = 0$.
Next, using Lemma \ref{lem;symbol_3}, we have

	\begin{align*}
		&\textup{Re}\langle \sigma_{\xi}(D_{(g,f,\psi)}Q_{\psi})(\eta,h, s),\, s\rangle\\[1ex]
		&= \textup{Re}\langle \sigma_{\xi}\left( D\Delta_g\psi \right)(\eta,h,s),\, s\rangle +  \textup{Re}\left\langle\sigma_{\xi}\left( D\mathcal{L}_{W(g,g_0)}^{\text{spin}}\psi \right)(\eta,h,s),\, s \right\rangle  -\frac{2}{\tau}\, \textup{Re}\left\langle \sigma_{\xi}\left(  D \mathcal{L}_{U}^{\text{spin}}\psi \right)(\eta,h,s),\,s \right\rangle\\[1ex]
		&=  -\frac{1}{4}\textup{Re}\langle \psi,\, c( \eta(\xi,\cdot)\wedge\xi) s\rangle - |\xi|^2 |s|^2 -\frac{1}{2}\,\textup{Re}\langle  \psi,\,c(\eta(\xi,\cdot)\wedge\xi) s\rangle \\[1ex]
		&\quad -\frac{2}{\tau}\, \left( -\frac{1}{8} |\psi|^2\,  \textup{Re}\langle \psi,\, c(\eta(\xi,\cdot)\wedge \xi)\, s\rangle + \frac{1}{2}  \sum_a\textup{Re} \langle\psi,\, e_a\cdot \xi\cdot s\rangle^2\right)\\[1ex]
		&= -   |\xi|^2 |s|^2 -\frac{1}{\tau}\sum_a \textup{Re}\langle\psi,\, e_a\cdot \xi\cdot s\rangle^2 - \left(\frac{3}{4}  - \frac{|\psi_0|^2}{4\tau}\right)\,\textup{Re}\langle  \psi,\,c(\eta(\xi,\cdot)\wedge\xi) s\rangle.
	\end{align*}
On the other hand, from the definition of $A(\psi)^{kl}$, we also have
\begin{align*}
	\textup{Re}\langle \sigma_{\xi}(D_{(g,f,\psi)}A(\psi)^{kl}\hat{\nabla}_k\hat{\nabla}_l\psi)(\eta,h, s),\, s\rangle= -   |\xi|^2 |s|^2 -\frac{1}{\tau}\sum_a \textup{Re}\langle\psi,\, e_a\cdot \xi\cdot s\rangle^2.
\end{align*}
This implies that $\sigma_{\xi}(D_{(g,f,\psi)}\mathcal{H}(g,f,\psi))(\eta,h, s) =  - \left(\frac{3}{4}  - \frac{|\psi_0|^2}{4\tau}\right)\,\textup{Re}\langle  \psi,\,c(\eta(\xi,\cdot)\wedge\xi) s\rangle$.
\end{proof}

\medskip

\subsection{Proof of short-time existence}

We solve the gauged system (\ref{eq:coupled_dg_gauged})-(\ref{eq:coupled_dpsi_gauged}) using a fixed point argument. To that end, Fix $(g_0, f_0, \psi_0)$ such that $|\psi_0|^2 \equiv c$. Let us consider the case where $c<\tau$ so that \eqref{eq:coupled_du_gauged} is backward parabolic. The proof of the case $c>\tau$ in which \eqref{eq:coupled_du_gauged} being forward parabolic is exactly the same. Let $T>0$ be chosen later. We consider the parabolic H\"older space
	
	\begin{align*}
		\mathcal{Y}_T &= C^{2+\alpha,1+\frac{\alpha}{2}}(M\times [0,T];\,\text{Sym}^2(T^*M)\oplus\R \oplus\mathcal{S}).
	\end{align*}
Moreover, let $\Lambda>0$ be a fixed constant to be chosen, we define closed subset in the parabolic H\"older space $\mathcal{Y}_T$ by
	
	\begin{align*}
		\mathcal{B}_{T,\Lambda} &:= \left\{ (g,f,\psi)\in \mathcal{Y}_T|\ g(\cdot,0)=g_0,\, \psi(\cdot,0)=\psi_0,\, f(\cdot,T) = f_0\ \text{and}\ \|(g,f,\psi)\|_{\mathcal{Y}_T}\leq \Lambda \right\}.
	\end{align*}

\begin{lemma}\label{lem:existence_decoupled_system}
	For each $(\omega, v,\varphi)\in \mathcal{Y}_T$, there exists a unique $(g,f,\psi)\in \mathcal{Y}_T$ solving the system
	
	\begin{align}
			\label{eq:decoupled_dg_gauged}\D{g}{t} - g_0^{kl}\hat{\nabla}_k\hat{\nabla}_lg & = (\omega^{kl}-g_0^{kl})\hat{\nabla}_k\hat{\nabla}_l\omega + \mathcal{F}(\omega, v,\varphi)\\[1ex]
			\label{eq:decoupled_du_gauged}\D{f}{t}+ \left( 1 - \frac{|\psi_0|^2}{\tau}\right)\Delta_{g_0} f & = -\left( 1 - \frac{|\psi_0|^2}{\tau}\right)(\Delta_{g} v - \Delta_{g_0} v)+    + \mathcal{G}(g, v,\psi)\\[1ex]
			\label{eq:decoupled_dpsi_gauged}\D{\psi}{t} - A(\psi_0)^{kl}\hat{\nabla}_k\hat{\nabla}_l\psi & = (A(\varphi)^{kl} - A(\psi_0)^{kl})\hat{\nabla}_k\hat{\nabla}_l\varphi + \mathcal{H}(g, v,\varphi)\\[1ex]
			\notag g(\cdot,0) &= g_0\\[1ex]
			\notag \psi(\cdot,0) &= \psi_0\\[1ex]
			\notag f(\cdot,T) &= f_0,
	\end{align}
where $\mathcal{F},\mathcal{G},\mathcal{H}$ are defined in (\ref{def:lower_order_terms}).  Moreover, if $(\omega, v,\varphi)\in \mathcal{B}_{T,\Lambda}$, then we have the estimate
\begin{align*}
	\|(g,f,\psi)\|_{\mathcal{Y}_T} \leq K
\end{align*}
for sufficiently small $T>0$, where $K = K(\|g_0\|_{2+\alpha;M}, \|f_0\|_{2+\alpha;M},\|\psi_0\|_{2+\alpha;M})$ is a positive constant. 
\end{lemma}

\begin{proof}
	Since (\ref{eq:decoupled_dg_gauged}) is a linear parabolic equation, there exists a unique solution $g\in C^{2+\alpha,1+\frac{\alpha}{2}}(M\times [0,T];\,\text{Sym}^2(T^*M))$ to (\ref{eq:decoupled_dg_gauged}). Invoking parabolic Schauder estimate, we have
	\begin{align*}
		&\|g\|_{2+\alpha,1+\frac{\alpha}{2};M\times[0,T]}\\[1ex]
		& \leq C \Big( \|(\omega^{kl}-g_0^{kl})\hat{\nabla}_k\hat{\nabla}_l\omega \|_{\alpha,\frac{\alpha}{2};M\times[0,T]} + \| \mathcal{F}(\omega,v,\varphi)\|_{\alpha,\frac{\alpha}{2};M\times[0,T]} +\|g_0\|_{2+\alpha;M}\Big).
	\end{align*}
	If $(\omega, v,\varphi)\in \mathcal{B}_{T,\Lambda}$, then using Lemma \ref{lem:auxiliary_estimate} we obtain
	\begin{equation}
		\begin{aligned}
			&\|g\|_{2+\alpha,1+\frac{\alpha}{2};M\times[0,T]}\\[1ex]
			&\leq  C \Big( \|\omega^{kl}-g_0^{kl}\|_{\alpha,\frac{\alpha}{2};M\times[0,T]} \|\hat{\nabla}_k\hat{\nabla}_l\omega \|_{\alpha,\frac{\alpha}{2};M\times[0,T]} + \| \mathcal{F}(g_0,f_0,\psi_0)\|_{\alpha,\frac{\alpha}{2};M\times[0,T]}\\[1ex]
			&\hspace{1cm} + \| \mathcal{F}(\omega, v,\varphi) - \mathcal{F}(g_0,f_0,\psi_0)\|_{\alpha,\frac{\alpha}{2};M\times[0,T]} +\|g_0\|_{2+\alpha;M}\Big)\\[1ex]
			&\leq C_1(\Lambda)\, T^{\frac{1}{2+\alpha}} + C_2,\\[1ex]
		\end{aligned}
	\end{equation}
	where $C_2 = C_2(\|g_0\|_{2+\alpha;M}, \|f_0\|_{2+\alpha;M},\|\psi_0\|_{2+\alpha;M})$. Choosing $T>0$ sufficiently small, we obtain
	
	\begin{equation}\label{estimate:decoupled_g}
		\begin{aligned}
			\|g\|_{2+\alpha,1+\frac{\alpha}{2};M\times[0,T]} \leq 2C_2.
		\end{aligned}
	\end{equation}
	 Next, having (\ref{eq:decoupled_dg_gauged}) solved, from the proof of Lemma \ref{lem:lower_order_terms}, we see that (\ref{eq:decoupled_dpsi_gauged}) is a linear parabolic system. Hence from standard parabolic theory (c.f. \cite{LSU1968}) there exists a unique solution $\psi\in C^{2+\alpha,1+\frac{\alpha}{2}}(M\times [0,T];\,\mathcal{S})$ to (\ref{eq:decoupled_dpsi_gauged}). Moreover, we can apply parabolic Schauder estimate to (\ref{eq:decoupled_dpsi_gauged}):
	\begin{align*}
		&\|\psi\|_{2+\alpha,1+\frac{\alpha}{2};M\times[0,T]}\\[1ex]
		& \leq C \Big( \|(A(\varphi)^{kl} - A(\psi_0)^{kl})\hat{\nabla}_k\hat{\nabla}_l\varphi \|_{\alpha,\frac{\alpha}{2};M\times[0,T]} + \| \mathcal{H}(g, v,\varphi)\|_{\alpha,\frac{\alpha}{2};M\times[0,T]} + \|\psi_0\|_{2+\alpha;M}\Big).
	\end{align*}
	If $(\omega, v,\varphi)\in \mathcal{B}_{T,\Lambda}$, then using Lemma \ref{lem:auxiliary_estimate} we obtain
	\begin{equation}
		\begin{aligned}
			&\|\psi\|_{2+\alpha,1+\frac{\alpha}{2};M\times[0,T]}\\[1ex]
			&\leq  C \Big( \|(A(\varphi)^{kl} - A(\psi_0)^{kl}\|_{\alpha,\frac{\alpha}{2};M\times[0,T]} \|\hat{\nabla}_k\hat{\nabla}_l\varphi  \|_{\alpha,\frac{\alpha}{2};M\times[0,T]} + \| \mathcal{H}(g,f_0,\psi_0)\|_{\alpha,\frac{\alpha}{2};M\times[0,T]}\\[1ex]
			&\hspace{1cm} + \| \mathcal{H}(g, v,\varphi) - \mathcal{H}(g,f_0,\psi_0)\|_{\alpha,\frac{\alpha}{2};M\times[0,T]} + \|\psi_0\|_{2+\alpha;M}\Big)\\[1ex]
			&\leq C_3(\Lambda)\, T^{\frac{1}{2+\alpha}} + C_4(\|g\|_{2+\alpha,1+\frac{\alpha}{2};M\times[0,T]}, \|f_0\|_{2+\alpha;M}, \|\psi_0\|_{2+\alpha;M})  \\[1ex]
			&\hspace{1cm} + C_5(\|g\|_{2+\alpha,1+\frac{\alpha}{2};M\times[0,T]}, \Lambda) \, T^{\frac{1}{2+\alpha}}.
		\end{aligned}
	\end{equation}
	Combining the above estimate with (\ref{estimate:decoupled_g}), we obtain, for sufficiently small $T>0$, 
	\begin{equation}\label{estimate:decoupled_psi}
		\begin{aligned}
			\|\psi\|_{2+\alpha,1+\frac{\alpha}{2};M\times[0,T]} \leq C_6,
		\end{aligned}
	\end{equation}
	where $C_6 = C_6(\|g_0\|_{2+\alpha;M}, \|f_0\|_{2+\alpha;M},\|\psi_0\|_{2+\alpha;M})$. Lastly, having (\ref{eq:decoupled_dg_gauged}) and (\ref{eq:decoupled_dpsi_gauged}) solved, we see that (\ref{eq:decoupled_du_gauged}) is a linear backward parabolic equation, in which a unique solution exists as long as the coefficients are defined.  Applying standard parabolic Schauder estimate to $f(T-t)$, and invoking Lemma \ref{lem:auxiliary_estimate} with (\ref{estimate:decoupled_g}) and (\ref{estimate:decoupled_psi}) we have
	\begin{equation}
		\begin{aligned}
			&\|f\|_{2+\alpha,1+\frac{\alpha}{2};M\times[0,T]}
			\\[1ex]
			 &\leq C \Big( \|\Delta_{g_0} v - \Delta_g v \|_{\alpha,\frac{\alpha}{2};M\times[0,T]} + C(\|g\|_{2+\alpha,1+\frac{\alpha}{2};M\times[0,T]}, \|\psi\|_{2+\alpha,1+\frac{\alpha}{2};M\times[0,T]}, \|f_0\|_{2+\alpha;M}) \Big)\\[1ex]
			 &\leq C_7(\Lambda)\, T^{\frac{1}{2+\alpha}} + C_8 (\|g_0\|_{2+\alpha;M}, \|f_0\|_{2+\alpha;M},\|\psi_0\|_{2+\alpha;M}).
		\end{aligned}
	\end{equation}
Hence	
	\begin{equation}
		\begin{aligned}
			\|f\|_{2+\alpha,1+\frac{\alpha}{2};M\times[0,T]} \leq C(\|g_0\|_{2+\alpha;M}, \|f_0\|_{2+\alpha;M},\|\psi_0\|_{2+\alpha;M}).
		\end{aligned}
	\end{equation}
for $T>0$ sufficiently small. 	This concludes the lemma.
\end{proof}

\bigskip

  Let us choose $\Lambda > K$, where $K$ is the constant given in Lemma \ref{lem:existence_decoupled_system}. Then by Lemma \ref{lem:existence_decoupled_system}, we can define a map
	\[	\mathcal{P}: \mathcal{Y}_T \to\mathcal{Y}_T\]
	as follows: for $(\omega, v,\varphi)\in \mathcal{Y}_T$, define $\mathcal{P}(\omega, v,\varphi) = (g,f,\psi)\in \mathcal{Y}_T$ to be the unique solution of the system (\ref{eq:decoupled_dg_gauged})-(\ref{eq:decoupled_dpsi_gauged}) considered in Lemma \ref{lem:existence_decoupled_system}. Moreover, Lemma \ref{lem:existence_decoupled_system} asserts that $\mathcal{P}(\mathcal{B}_{T,\Lambda}) \subset \mathcal{B}_{T,\Lambda}$ provided $T>0$ sufficiently small.

\medskip
\begin{lemma}
	Let $(\omega_1, v_1,\varphi_1), (\omega_2, v_2,\varphi_2)\in\mathcal{B}_{T,\Lambda}$. Then we have the estimate
	\begin{align*}
		\|\mathcal{P}(\omega_1, v_1,\varphi_1) - \mathcal{P}(\omega_2, v_2,\varphi_2)\|_{\mathcal{Y}_T} \leq C(\Lambda,g_0,f_0,\psi_0) T^{\frac{1}{2+\alpha}} \|(\omega_1, v_1,\varphi_1) - (\omega_2, v_2,\varphi_2)\|_{\mathcal{Y}_T}
	\end{align*}
	for $T>0$ sufficiently small.
\end{lemma}

\begin{proof}
	Let us denote $(g_i, f_i,\psi_i) = \mathcal{P}(\omega_i,v_i,\varphi_i)$ for $i=1,2$, and let $\delta g = g_1-g_2$, $\delta f = f_1-f_2$ and $\delta\psi = \psi_1-\psi_2$. 	 Then $(\delta g, \delta f, \delta\psi)$ solves the PDE system
	
		\begin{align}
			\label{eq:decoupled_delta_g_gauged}\D{\delta g}{t} - g_0^{kl}\hat{\nabla}_k\hat{\nabla}_l \delta g & = (\omega_1^{kl}-g_0^{kl})\hat{\nabla}_k\hat{\nabla}_l(\omega_1 - \omega_2) + (\omega_1^{kl}-\omega_2^{kl})\hat{\nabla}_k\hat{\nabla}_l\omega_2\\
			&\notag\quad  + \mathcal{F}(\omega_1, v_1,\varphi_1) - \mathcal{F}(\omega_2, v_2,\varphi_2)\\[1ex]
			\label{eq:decoupled_delta_u_gauged}\D{\delta f}{t}+ \left( 1 - \frac{|\psi_0|^2}{\tau}\right)\Delta_{g_0} \delta f & = -\left( 1 - \frac{|\psi_0|^2}{\tau}\right)(\Delta_{g_1}  - \Delta_{g_0} )(v_1 - v_2)    -\left( 1 - \frac{|\psi_0|^2}{\tau}\right)(\Delta_{g_1}  - \Delta_{g_2} ) v_2 \\
			&\notag\quad + \mathcal{G}(g_1,f_1,\psi_1) -  \mathcal{G}(g_2,f_2,\psi_2)\\[1ex]
			\label{eq:decoupled_delta_psi_gauged}\D{\delta \psi}{t} - A(\psi_0)^{kl}\hat{\nabla}_k\hat{\nabla}_l\delta \psi & = (A(\varphi_1)^{kl} - A(\psi_0)^{kl})\hat{\nabla}_k\hat{\nabla}_l (\varphi_1 - \varphi_2) + (A(\varphi_1)^{kl} - A(\varphi_2)^{kl})\hat{\nabla}_k\hat{\nabla}_l \varphi_2\\
			&\notag\quad  + \mathcal{H}(g_1, v_1,\varphi_1) -  \mathcal{H}(g_2,v_2,\varphi_2)\\[1ex] 
			\notag \delta g(\cdot,0) &= 0\\[1ex]
			\notag \delta \psi(\cdot,0) &= 0\\[1ex]
			\notag \delta f(\cdot,T) &= 0.
	\end{align}
 Applying parabolic Schauder esimtate to (\ref{eq:decoupled_delta_g_gauged}), we get

\begin{equation}
	\begin{aligned}
		&\|\delta g\|_{2+\alpha,1+\frac{\alpha}{2};M\times[0,T]}\\[1ex]
		& \leq C(\Lambda) \Big(  \|\omega_1-g_0\|_{\alpha,\frac{\alpha}{2};M\times[0,T]}\|\omega_1 - \omega_2\|_{2+\alpha,1+\frac{\alpha}{2};M\times[0,T]} + \|\omega_1 - \omega_2\|_{\alpha,\frac{\alpha}{2};M\times[0,T]}\|\omega_2\|_{2+\alpha,1+\frac{\alpha}{2};M\times[0,T]}\\[1ex]
		&\hspace{1.5cm} + \|\mathcal{F}(\omega_1, v_1,\varphi_1) - \mathcal{F}(\omega_2, v_2,\varphi_2)\|_{\alpha,\frac{\alpha}{2};M\times[0,T]} \Big).
	\end{aligned}
\end{equation}
By noting that $\omega_i(\cdot,0) = g_0$, we can apply  Lemma \ref{lem:auxiliary_estimate} to the above Schauder estimate, we subsequently obtain

\begin{equation}\label{estimate:delta_g}
	\begin{aligned}
		\|\delta g\|_{2+\alpha,1+\frac{\alpha}{2};M\times[0,T]} \leq C(\Lambda) T^{\frac{1}{2+\alpha}}  \|(\omega_1,v_1,\varphi_1) - (\omega_2,v_2,\varphi_2)\|_{\mathcal{Y}_T} .
	\end{aligned}
\end{equation}
\medskip

Next, we note that $\mathcal{H}(g, v,\varphi)$ depends on top order of $g$, but only up to first order of $v,\varphi$. Similarly, we can apply parabolic Schauder estimate with Lemma \ref{lem:auxiliary_estimate} to (\ref{eq:decoupled_delta_psi_gauged}), we obtain

\begin{equation}\label{estimate:delta_psi}
	\begin{aligned}
		&\|\delta \psi\|_{2+\alpha,1+\frac{\alpha}{2};M\times[0,T]} \\[1ex]
		&\leq C(\Lambda) 	\|\delta g\|_{2+\alpha,1+\frac{\alpha}{2};M\times[0,T]} + C(\Lambda) T^{\frac{1}{2+\alpha}}\|(\omega_1,v_1,\varphi_1) - (\omega_2,v_2,\varphi_2)\|_{\mathcal{Y}_T}.
	\end{aligned}
\end{equation}
Here the extra term of $\|g_1 - g_2\|_{2+\alpha,1+\frac{\alpha}{2};M\times[0,T]}$ in the RHS of (\ref{estimate:delta_psi}) comes from the fact that $\mathcal{H}(g,v,\varphi)$ depends on second order in $g$. Putting (\ref{estimate:delta_g}) and (\ref{estimate:delta_psi}) together, we obtain

\begin{equation}\label{estimate:delta_g_and_psi}
	\begin{aligned}
	&\|\delta g\|_{2+\alpha,1+\frac{\alpha}{2};M\times[0,T]} + \|\delta \psi\|_{2+\alpha,1+\frac{\alpha}{2};M\times[0,T]}\leq C(\Lambda) T^{\frac{1}{2+\alpha}}\|(\omega_1,v_1,\varphi_1) - (\omega_2,v_2,\varphi_2)\|_{\mathcal{Y}_T}.\\[1ex]
	\end{aligned}
\end{equation}

Lastly, we apply parabolic Schauder estimate with Lemma \ref{lem:auxiliary_estimate} to (\ref{eq:decoupled_delta_u_gauged}), and noting that $\mathcal{G}(g,v,\varphi)$ depends on top order of $g,\varphi$, but only up to first order of $v$. We thus obtain

\begin{equation}\label{estimate:delta_u}
	\begin{aligned}
		&\|\delta u\|_{2+\alpha,1+\frac{\alpha}{2};M\times[0,T]} \\[1ex]
		&\leq C(\Lambda) 	\Big( \|\delta g\|_{2+\alpha,1+\frac{\alpha}{2};M\times[0,T]} + \|\delta \psi\|_{2+\alpha,1+\frac{\alpha}{2};M\times[0,T]}\Big) + C(\Lambda) T^{\frac{1}{2+\alpha}}\|(\omega_1,v_1,\varphi_1) - (\omega_2,v_2,\varphi_2)\|_{\mathcal{Y}_T}\\[1ex]
		&\leq C(\Lambda) T^{\frac{1}{2+\alpha}}\|(\omega_1,v_1,\varphi_1) - (\omega_2,v_2,\varphi_2)\|_{\mathcal{Y}_T}.
	\end{aligned}
\end{equation}
This proves the lemma.

\end{proof}
From this Lemma we obtain

\begin{corollary}
	The map $\mathcal{P} :\mathcal{B}_{T,\Lambda} \to \mathcal{B}_{T,\Lambda}$ is a contraction mapping  provided $T>0$ is sufficiently small.
\end{corollary}

\medskip

\begin{theorem}
	For $T = T(g_0, f_0,\psi_0)>0$ sufficiently small, the system (\ref{eq:coupled_dg})-(\ref{eq:coupled_dpsi}) has a solution.
\end{theorem}

\begin{proof}
	Let $(g,f,\psi)\in \mathcal{B}_{T,\Lambda}$ be the fixed point of $\mathcal{P}$. Then from the definition of $\mathcal{P}$, they solves the PDE system
	\begin{equation}
		\begin{aligned}
			\D{g}{t}  & =  g^{kl}\hat{\nabla}_k\hat{\nabla}_lg + \mathcal{F}(g,f,\psi)\\[1ex]
			\D{f}{t} & = -\left( 1 - \frac{|\psi_0|^2}{\tau}\right)\Delta_{g}f+   \mathcal{G}(g, f,\psi)\\[1ex]
			\D{\psi}{t}  & = A(\psi)^{kl}\hat{\nabla}_k\hat{\nabla}_l\psi + \mathcal{H}(g, f,\psi)\\[1ex]
			\notag g(\cdot,0) &= g_0\\[1ex]
			\notag \psi(\cdot,0) &= \psi_0\\[1ex]
			\notag f(\cdot,T) &= f_0,
		\end{aligned}
	\end{equation}
By the definition of $\mathcal{F},\mathcal{G},\mathcal{H}$ in \ref{def:lower_order_terms}, we see that $(g,f,\psi)$ is a solution to the system (\ref{eq:coupled_dg_gauged})-(\ref{eq:coupled_dpsi_gauged}). This is because $|\psi|^2 = |\psi_0|^2$ is preserved by (\ref{eq:coupled_dpsi_gauged}). Lastly, let us define a vector field
	\[	X = W(g,g_0) - \frac{2}{\tau}V_f + \nabla f,\]
	and let $\Phi_t:M\to M$ be the one-parameter family of diffeomorphisms generated by
	\begin{align*}
		\begin{cases}
			&\frac{\partial}{\partial t}\Phi_t = -X\circ\Phi_t\\
			&\Phi_0 = \id
		\end{cases}.
	\end{align*}
	Then $(\tilde{g},\tilde{f},\tilde{\psi}) = (\Phi_t^*g, \Phi_t^*f, \Phi_t^*\psi)$ is a solution of (\ref{eq:coupled_dg})-(\ref{eq:coupled_dpsi}). For example, as $V_f = U + \frac{|\psi|^2}{2}\nabla f$, so $\div(V_f) = \div\, U + \frac{|\psi|^2}{2}\Delta f$, and 
	
	\begin{align*}
		\D{\tilde{f}}{t} &= \Phi_t^* \D{f}{t} - \langle\nabla \tilde{f}, X\circ\Phi_t\rangle\\
		& =  \Phi_t^*\Big( -\Delta f - R + \frac{\lambda n}{2\tau} +\frac{4}{\tau} \Tr_g S_{g,f,\psi} + \frac{2}{\tau}\div\, V_f  + \langle \nabla f, W - \frac{2}{\tau}V_f + \nabla f\rangle\Big)\\
		&\quad - \langle\nabla \tilde{f}, (W - \frac{2}{\tau}V_f + \nabla f)\circ\Phi_t\rangle\\
		&=\Phi_t^*\Big( -\Delta f - R + \frac{\lambda n}{2\tau} +\frac{4}{\tau} \Tr_g S_{g,f,\psi} + \frac{2}{\tau}\div\, V_f\Big).
	\end{align*}
\end{proof}

\medskip

\begin{remark}
The short-time existence result provides the analytic starting point for studying the long-time
behavior of the spinorial Ricci flow as a dynamical system on the constraint space
$\{(g,f,\psi): \int_M d\Omega_{g,f,\tau}=1,\ |\psi|^2\equiv c\}$ modulo diffeomorphisms. In the forward
parabolic regime $c>\tau$, the DeTurck-gauged system is a genuine forward quasilinear parabolic flow in
all variables, and the dissipation identity \eqref{eq:monotonicity_A} makes $\mathbb{W}_\lambda$ a strict Lyapunov
functional. This suggests the possibility of applying tools from analytic gradient-flow theory
(e.g.\ {\L}ojasiewicz--Simon inequalities) to obtain stability and convergence to critical points
(twisted solitons/eigenspinors), once suitable compactness and gauge-fixing inputs are available.

It is also natural to view the classical Ricci flow as an invariant subsystem: when the initial spinor is
$f$-harmonic, the spinorial terms vanish and the metric evolution reduces (up to diffeomorphisms) to the
Ricci flow. Thus one may ask whether the $f$-harmonic locus is dynamically stable inside the larger
spinorial flow, and whether $\mathbb{W}_\lambda$ can be used to extract additional monotonicity or rigidity
information in this setting. Finally, in the mixed regime $c<\tau$ the scalar equation becomes backward
parabolic after gauging, reminiscent of the adjoint structure in Perelman's coupled system; it would be
interesting to understand how this regime can be suited for entropy/singularity applications in
which $f$ plays the role of an adjoint variable rather than an evolving state.
\end{remark}
\medskip

\begin{remark}
Theorem~\ref{thm;short-time-existence} yields a well-posed local evolution for the spinorial Ricci flow,
and hence (by continuation) a maximal-time solution in the usual quasilinear-parabolic sense. The
distinction between $c<\tau$ and $c>\tau$ is not merely technical: in the regime $c>\tau$ the DeTurck-gauged
system is a \emph{fully forward} quasilinear parabolic flow in $(g,f,\psi)$, while $\mathbb{W}_\lambda$ is a strict
Lyapunov functional by the dissipation identity \eqref{eq:monotonicity_A}. In particular, for any global
solution with $\mathbb{W}_\lambda$ bounded from below, integrating \eqref{eq:monotonicity_A} shows that the squared
$L^2(d\Omega)$-norm of the constrained gradient is integrable in time, and hence there exist times
$t_j\to\infty$ along which the flow approaches the critical set (twisted solitons/eigenspinors) in the
sense that the $L^2$-gradient tends to zero. A natural next step is to understand when such
subsequential convergence improves to full convergence modulo diffeomorphisms; in the forward-parabolic
setting one may hope to combine compactness/regularity inputs with a {\L}ojasiewicz--Simon inequality
for $\mathbb{W}_\lambda$ near an isolated critical point, as is standard for analytic geometric gradient flows.
\end{remark}

\bigskip

\appendix

\section{Auxiliary estimate}

	\begin{lemma}\label{lem:auxiliary_estimate}
		Suppose that $E\to M$ is a vector bundle and $F:E\oplus (T^*M\otimes E)\to\R$ is a fibrewise $C^2$ map. Suppose that $\omega_1, \omega_2\in C^{2+\alpha, 1+\frac{\alpha}{2}}(M\times [0,T];E)$ with $\omega_1(\cdot,0) = \omega_2(\cdot,0)$, or $\omega_1(\cdot,T) = \omega_2(\cdot,T)$. Then		
		\begin{align*}
			&\|F(\omega_1,\partial\omega_1) - F(\omega_2,\partial\omega_2)\|_{\alpha,\frac{\alpha}{2};M\times[0,T]}\\
			& \leq C(F,  \|\omega_i\|_{2+\alpha, 1+\frac{\alpha}{2}; M\times [0,T]})\, T^{\frac{1}{2+\alpha}}\, \|\omega_1-\omega_2\|_{2+\alpha, 1+\frac{\alpha}{2}; M\times [0,T]}
		\end{align*}
		for small $T>0$. Moreover, if $F(0,0) = 0$, then		
		\begin{align*}
			\|F(\omega,\partial\omega)\|_{\alpha,\frac{\alpha}{2};M\times[0,T]} \leq C(F,  \|\omega\|_{2+\alpha, 1+\frac{\alpha}{2}; M\times [0,T]})\, T^{\frac{1}{2+\alpha}}
		\end{align*}
		for $\omega\in C^{2+\alpha, 1+\frac{\alpha}{2}}(M\times [0,T];E)$ satisfying $\omega(\cdot,0) = 0$, or $\omega(\cdot,T) = 0$.
	\end{lemma}

\begin{proof}
	Using the mean value theorem, we have
	\begin{equation}\label{appendix:eq1}
		\begin{aligned}
		&\|F(\omega_1,\partial\omega_1) - F(\omega_2,\partial\omega_2)\|_{\alpha,\frac{\alpha}{2};M\times[0,T]} \\
		&\leq C(F,  \|\omega_i\|_{2+\alpha, 1+\frac{\alpha}{2}; M\times [0,T]})\, \Big( \|\omega_1-\omega_2\|_{\alpha,\frac{\alpha}{2};M\times[0,T]} + \|\partial\omega_1- \partial \omega_2\|_{\alpha,\frac{\alpha}{2};M\times[0,T]} \Big).
		\end{aligned}
	\end{equation}
	Let us take $u = \omega_1 - \omega_2$, then $u\in C^{2+\alpha, 1+\frac{\alpha}{2}}(M\times [0,T];E)$ and the assumption implies $u(\cdot,0) = 0$. Using the parabolic H\"older interpolation inequalities (c.f. \cite[Chapter 8.8]{Krylov1996}), we have 
	\begin{equation}\label{appendix:eq2}
		\begin{aligned}
			 \,[u]_{\alpha,\frac{\alpha}{2};M\times[0,T]}  &\leq C \|u\|_{2+\alpha, 1+\frac{\alpha}{2}; M\times [0,T]}^{\frac{\alpha}{2+\alpha}}\, |u|_{0; M\times [0,T]}^{1-\frac{\alpha}{2+\alpha}},\\[1ex]
			|\partial u|_{0; M\times [0,T]}  &\leq C \|u\|_{2+\alpha, 1+\frac{\alpha}{2}; M\times [0,T]}^{\frac{1}{2+\alpha}}\, |u|_{0; M\times [0,T]}^{1-\frac{1}{2+\alpha}},\\[1ex]
			[\partial u]_{\alpha,\frac{\alpha}{2};M\times[0,T]}  &\leq C \|u\|_{2+\alpha, 1+\frac{\alpha}{2}; M\times [0,T]}^{\frac{1+\alpha}{2+\alpha}}\, |u|_{0; M\times [0,T]}^{1-\frac{1+\alpha}{2+\alpha}}.
		\end{aligned}
	\end{equation}
	On the other hand, using the fact that $u(\cdot,0) = 0$, we have
	\[	|u(x,t)| \leq \int_0^t |u_t(x,s)| ds \leq t\, |u_t|_{0;M\times[0,T]}\leq T\,\|u\|_{2+\alpha, 1+\frac{\alpha}{2}; M\times [0,T]} \]
	for all $(x,t)\in M\times [0,T]$. This gives
	\begin{equation}\label{appendix:eq3}
		|u|_{0; M\times [0,T]} \leq T\,\|u\|_{2+\alpha, 1+\frac{\alpha}{2}; M\times [0,T]}.
	\end{equation}
	The above estimate holds similarly if we have $u(\cdot,T) = 0$ instead. Lastly, putting (\ref{appendix:eq2}) and   (\ref{appendix:eq3}) into (\ref{appendix:eq1}), we obtain 
	\begin{align*}
		&\|F(\omega_1,\partial\omega_1) - F(\omega_2,\partial\omega_2)\|_{\alpha,\frac{\alpha}{2};M\times[0,T]}\\[1ex]
		 &\leq  C(F,  \|\omega_i\|_{2+\alpha, 1+\frac{\alpha}{2}; M\times [0,T]})\, T^{\min\{1-\frac{\alpha}{2+\alpha}, 1-\frac{1}{2+\alpha}, 1-\frac{1+\alpha}{2+\alpha}, 1 \}}\, \|u\|_{2+\alpha, 1+\frac{\alpha}{2}; M\times [0,T]}\\[1ex]
		&\leq C(F,  \|\omega_i\|_{2+\alpha, 1+\frac{\alpha}{2}; M\times [0,T]})\, T^{\frac{1}{2+\alpha}}\, \|\omega_1-\omega_2\|_{2+\alpha, 1+\frac{\alpha}{2}; M\times [0,T]}
	\end{align*}
	for small $T>0$.
\end{proof}

\bibliographystyle{plain}  % or alpha, abbrv, unsrt, etc.
\bibliography{citations.bib}  % assumes references.bib

\end{document}